\documentclass{elsarticle}
\usepackage{amsmath, amsfonts, amssymb, epsfig, amscd, graphicx, latexsym, latexsym, amsthm, etex, synttree, tikz}
\usepackage[latin5]{inputenc}
\usepackage{young}

\newcommand\Inv{{\mathrm{In}}}
\newcommand\D{{\mathcal{D}}}
\newcommand\BL{{\mathrm{BL}}}
\newcommand\IBL{{\mathrm{IBL}}}
\newcommand\CBL{{\mathrm{CBL}}}

\newcommand\cp{{\mathrm{cp}}}

\newcommand\Des{{\mathrm{Des}}}

\newcommand{\beeq}{\begin{equation}}
\newcommand{\eneq}{\end{equation}}
\newcommand{\bece}{\begin{center}}
\newcommand{\ence}{\end{center}}

\newtheorem{thm}{Theorem}[section]
\newtheorem{pro}[thm]{Proposition}
\newtheorem{cor}[thm]{Corollary}
\newtheorem{lem}[thm]{Lemma}

\newtheorem{defn}[thm]{Definition}
\theoremstyle{definition}
\newtheorem{ex}[thm]{Example}
\newenvironment{rem}[1][Remark.]{\begin{trivlist}
\item[\hskip \labelsep {\bfseries #1}]}{\end{trivlist}}

\newcommand{\cellsize}{10}
\newlength{\cellsz} 
\newsavebox{\cell}
\sbox{\cell}{\begin{picture}(\cellsize,\cellsize)
\put(0,0){\line(1,0){\cellsize}}
\put(0,0){\line(0,1){\cellsize}}
\put(\cellsize,0){\line(0,1){\cellsize}}
\put(0,\cellsize){\line(1,0){\cellsize}}
\end{picture}}
\newcommand\cellify[1]{\def\thearg{#1}\def\nothing{}%
\ifx\thearg\nothing
\vrule width0pt height\cellsz depth0pt\else
\hbox to 0pt{\usebox{\cell} \hss}\fi%
\vbox to \cellsz{
\vss
\hbox to \cellsz{\hss$#1$\hss}
\vss}}
\newcommand\tableau[1]{\vtop{\let\\\cr
\baselineskip -16000pt \lineskiplimit 16000pt \lineskip 0pt
\ialign{&\cellify{##}\cr#1\crcr}}}

\begin{document}
\begin{frontmatter}
\title{Tower tableaux and Schubert polynomials}
\author{Olcay Co\c{s}kun}
\ead{olcay.coskun@boun.edu.tr}

\author{M\"uge Ta\c{s}k\i n}
\ead{muge.taskin@boun.edu.tr}

\address{Boğaziçi Üniversitesi, Matematik B\"ol\"um\"u 34342 Bebek,  Istanbul, Turkey.}

\fntext[fn1]{Both of the authors are supported by Boğaziçi-BAP-6029 and by Tübitak/3501/109T661.}

\begin{abstract}
We prove that the well-known condition of being a balanced labeling can be characterized in terms of the
sliding algorithm on tower diagrams. The characterization involves a generalization of authors' Rothification
algorithm. Using the characterization, we obtain descriptions of Schubert polynomials and Stanley symmetric
functions.
\end{abstract}
\end{frontmatter}
\section{Introduction}\label{Section:introduction}
Tower diagrams are introduced by the authors as a new way to study reduced words of permutations.
The same diagrams are also studied by several authors in different contents, see \cite{BH}, \cite{H}, \cite{RY}.
In this paper, our aim is to study the relation between the tower tableaux and Schubert polynomials.
A relation between these objects can be predicted since the well-known Rothe diagrams are closely related to
tower diagrams, as shown in \cite[Section 6 - 7]{CT} as well as to Schubert polynomials as shown
in \cite{FGRS}.

Given a permutation $\omega\in S_n$, recall that, the Rothe diagram of $\omega$ is equivalent to the diagram of 
inversions
and is constructed by removing certain hooks from a square diagram of size $n$. On the other hand,
the tower diagram of $\omega$ is obtained by sliding a reduced word of $\omega$ to the empty diagram, with the
rules recalled in Section \ref{Section:prelim}. By Theorem 2.4 in \cite{FGRS}, there is a bijection between all
reduced words of the permutation $\omega$ and all \textit{injective balanced} labelings of its Rothe diagram. On
the other hand, by Theorem 4.4 in \cite{CT}, there is a bijection between the set of all reduced words of $\omega$
and all \textit{standard} labelings of its tower diagram. See Section \ref{Section:prelim} for further descriptions.

This similarity between the two construction comes from the above mentioned characterization of tower and
Rothe diagrams, each in terms of the other. Precisely,
the tower diagram of $\omega$ can be determined by pushing the cells in the
Rothe diagram of $\omega$ to the top border of the diagram and then reflecting them on this border. Obtaining
the Rothe diagram from the tower diagram is more complicated and is given by the Rothification algorithm, also
recalled below. This algorithm makes use of a special labeling of tower diagrams.

In this paper, we push this relation between tower and Rothe diagrams forward to obtain a new description of
Schubert polynomials in terms of tower tableaux. In order to achieve this, we use the description of
Fomin-Greene-Reiner-Shimonozo which employs balanced labeling of Rothe diagrams.
Our main observation is that the Rothification algorithm can be extended to all standard labelings of tower
diagrams. This extension allows us to establish a bijective correspondence between standard tower tableaux
and injective balanced labelings of Rothe diagrams, and hence we transform the condition of being injective and
balanced to the condition of being standard.

We can rephrase this result as follows. Let STT$(\omega)$ denote the set of all standard tower tableaux of
shape $\mathcal T(\omega)$ and let IBL$(\omega)$ denote the set of all injective balanced labeling of the Rothe
diagram of $\omega$. Also let Red$(\omega)$ denote the set of all reduced words of $\omega$. Then the
following diagram commutes.
\begin{center}
\begin{picture}(160,70)
\put(-10,55){STT$(\omega)$}
\put(160,55){IBL$(\omega)$}
\put(85,5){Red$(\omega)$}
\put(10,20){Sliding}
\put(150,20){Canonical tableau}
\put(70,58){Rothification}

\multiput(30,56)(5,0){25}{\line(1,0){5}}
\multiput(150,56)(5,0){1}{\vector(1,0){3}}
\multiput(80,10)(5,0){1}{\vector(-3,2){55}}

\multiput(120,10)(5,0){1}{\vector(4,3){50}}
\end{picture}
\end{center}
This bijection suggests the existence of a more general correspondence between balanced labelings of Rothe
diagrams and certain labelings of tower diagrams. We show, in Section \ref{Section:BalancedLabel}, that the
Rothification algorithm  can be extended to \textit{semi-standard} tower tableaux in such a way that the above
bijection between standard tower tableaux and injective balanced labelings extends to a bijection between
semi-standard tableaux and balanced labelings. Now, being balanced is transformed to being semi-standard.

The reason behind the above correspondence can be seen by determining the hooks in tower diagrams that
correspond to the hooks in Rothe diagrams under Rothification. We include this at the end of Section
\ref{Section:BalancedLabel}.

Returning back to Schubert polynomials, Fomin-Greene-Reiner-Shimonozo prove that the monomials in the
Schubert polynomial of a permutation $\omega$ are determined by flagged column strict balanced labelings of
its Rothe diagram. Using the above correspondence, we prove, in Section \ref{Section:Schubert}, that the
monomials can also be the described by the flagged semi-standard tower tableaux. As in the case of balanced
labelings, our result also describes Stanley symmetric functions, as indicated in Section \ref{Section:Schubert}.

\textbf{Acknowledgement.} We thank the referees for careful reading and helpful remarks.

\section{Preliminaries}\label{Section:prelim}
\subsection{Digression on tower diagrams}
In this section, we recall the necessary background from \cite{CT} without details.
To begin with, by a \textbf{\textit {tower}} $\mathcal{T}$ of size $k\ge 0$ we mean a vertical strip
of $k$ squares of side length $1$. Then a \textbf{\textit{tower diagram}} $\mathcal T$ is a finite sequence 
$(\mathcal T_1,\mathcal T_2,\ldots,\mathcal T_l)$ of towers. We always consider the tower diagram $\mathcal T$ 
as located on the first
 quadrant of the plane so that for each $i$, the tower $\mathcal{T}_i$ is located on the interval $[i-1,i]$ of the
 horizontal axis and has size equal to the size of $\mathcal T_i$. The following
 is an example of a tower diagram.
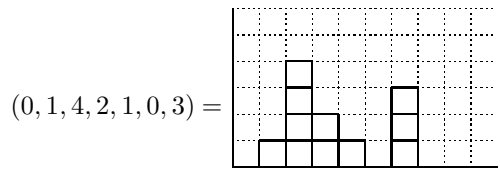
\begin{center}
\begin{picture}(100,60)
\put(-84,20){$(0,1,4,2,1,0,3)= $}
\multiput(0,0)(0,0){1}{\line(1,0){100}}
\multiput(0,0)(0,0){1}{\line(0,1){60}}
\multiput(0,10)(2,0){50}{\line(0,1){.1}}
\multiput(0,20)(2,0){50}{\line(0,1){.1}}
\multiput(0,30)(2,0){50}{\line(0,1){.1}}
\multiput(0,40)(2,0){50}{\line(0,1){.1}}
\multiput(0,50)(2,0){50}{\line(0,1){.1}}
\multiput(0,60)(2,0){50}{\line(0,1){.1}}
\multiput(10,0)(0,2){30}{\line(1,0){.1}}
\multiput(20,0)(0,2){30}{\line(1,0){.1}}
\multiput(30,0)(0,2){30}{\line(1,0){.1}}
\multiput(40,0)(0,2){30}{\line(1,0){.1}}
\multiput(50,0)(0,2){30}{\line(1,0){.1}}
\multiput(60,0)(0,2){30}{\line(1,0){.1}}
\multiput(70,0)(0,2){30}{\line(1,0){.1}}
\multiput(80,0)(0,2){30}{\line(1,0){.1}}
\multiput(90,0)(0,2){30}{\line(1,0){.1}}
\multiput(80,0)(0,2){30}{\line(1,0){.1}}
\multiput(90,0)(0,2){30}{\line(1,0){.1}} \put(10,0){\tableau{{}}}
\put(20,30){\tableau{{}\\{}\\{}\\{}}} \put(30,10){\tableau{{}\\{}}}
\put(40,0){\tableau{{}}} \put(60,20){\tableau{{}\\{}\\{}}}
\end{picture}
\end{center}

To any tower diagram $\mathcal{T}$, one can associate a set, still
denoted by $\mathcal{T}$, consisting of the pairs of non-negative
integers with the rule  that each pair $(i,j)$ corresponds to the
cell in $\mathcal{T}$ whose south-east corner is located at the
point $(i,j)$ of the first quadrant. Such a set can also be
characterized by the  rule that if $(i,j)\in \mathcal{T}$ then
$\{(i,0), (i,1),\ldots, (i,j) \}\subset \mathcal{T}$. For the rest,
we identify any square with its south-east corner.

There are two basic operations on tower diagrams. One of them, the
flight algorithm, is a way to decrease the number of cells in a
tower diagram. With this algorithm, a cell can be removed if it is a
corner cell. In practice, we choose a cell, say  $c$, from a tower
and move it in the north-west direction starting from  the line
passing through its main diagonal subject to the following
conditions. 
\begin{enumerate}
\item[(i)] If the line does not intersect any other cell, we say that the cell $c$ has a flight path consisting of just 
itself. 
\item[(ii)] Otherwise, if the first intersection with the tower diagram is the north-east corner of the top cell of a
tower, then we say that the cell $c$ has no flight path. 
\item[(iii)] Finally, if the intersection is through the main diagonal of another cell, say $d$, in the tower diagram, 
then we 
say that the cell $c$ has a flight path if and only if  the cell $e$ just below $d$ has. In that case the flight path of 
$c$ is the union of the flight path of $e$ and the cells $c$ and $d$. 
\end{enumerate}
A top cell of a tower with a flight path
is called a \textbf{corner} of the tower diagram. To each corner cell, we associate a \textbf{flight number} $f$ if the 
main diagonal of the left most cell in the flight path lies on the diagonal $x+y=f$. We denote by $c\nwarrow 
\mathcal T$ the tower diagram obtained by removing the corner cell $c$.

\begin{ex} In the following tower diagram, we illustrate the cells which have no flight path. Thus all other cells 
have a flight path.
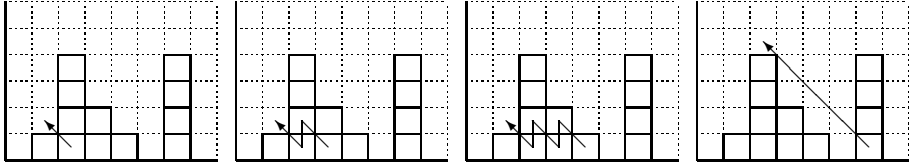
\begin{center}
\begin{picture}(80,60)
\multiput(0,0)(0,0){1}{\line(1,0){80}}
\multiput(0,0)(0,0){1}{\line(0,1){60}}
\multiput(0,10)(2,0){40}{\line(0,1){.1}}
\multiput(0,20)(2,0){40}{\line(0,1){.1}}
\multiput(0,30)(2,0){40}{\line(0,1){.1}}
\multiput(0,40)(2,0){40}{\line(0,1){.1}}
\multiput(0,50)(2,0){40}{\line(0,1){.1}}
\multiput(0,60)(2,0){40}{\line(0,1){.1}}
\multiput(10,0)(0,2){30}{\line(1,0){.1}}
\multiput(20,0)(0,2){30}{\line(1,0){.1}}
\multiput(30,0)(0,2){30}{\line(1,0){.1}}
\multiput(40,0)(0,2){30}{\line(1,0){.1}}
\multiput(50,0)(0,2){30}{\line(1,0){.1}}
\multiput(60,0)(0,2){30}{\line(1,0){.1}}
\multiput(70,0)(0,2){30}{\line(1,0){.1}}
\multiput(80,0)(0,2){30}{\line(1,0){.1}}
 \put(10,0){\tableau{{}}}
\put(20,30){\tableau{{}\\{}\\{}\\{}}} \put(30,10){\tableau{{}\\{}}}
\put(40,0){\tableau{{}}} \put(60,30){\tableau{{}\\{}\\{}\\{}}}
\multiput(25,5)(0,0){1}{\vector(-1,1){10}}
\end{picture}\hskip.1in
\begin{picture}(80,60)
\multiput(0,0)(0,0){1}{\line(1,0){80}}
\multiput(0,0)(0,0){1}{\line(0,1){60}}
\multiput(0,10)(2,0){40}{\line(0,1){.1}}
\multiput(0,20)(2,0){40}{\line(0,1){.1}}
\multiput(0,30)(2,0){40}{\line(0,1){.1}}
\multiput(0,40)(2,0){40}{\line(0,1){.1}}
\multiput(0,50)(2,0){40}{\line(0,1){.1}}
\multiput(0,60)(2,0){40}{\line(0,1){.1}}
\multiput(10,0)(0,2){30}{\line(1,0){.1}}
\multiput(20,0)(0,2){30}{\line(1,0){.1}}
\multiput(30,0)(0,2){30}{\line(1,0){.1}}
\multiput(40,0)(0,2){30}{\line(1,0){.1}}
\multiput(50,0)(0,2){30}{\line(1,0){.1}}
\multiput(60,0)(0,2){30}{\line(1,0){.1}}
\multiput(70,0)(0,2){30}{\line(1,0){.1}}
\multiput(80,0)(0,2){30}{\line(1,0){.1}}
 \put(10,0){\tableau{{}}}
\put(20,30){\tableau{{}\\{}\\{}\\{}}} \put(30,10){\tableau{{}\\{}}}
\put(40,0){\tableau{{}}} \put(60,30){\tableau{{}\\{}\\{}\\{}}}
\multiput(35,5)(0,0){1}{\line(-1,1){10}}
\multiput(25,15)(0,0){1}{\line(0,-1){10}}
\multiput(25,5)(0,0){1}{\vector(-1,1){10}}
\end{picture}\hskip.1in
\begin{picture}(80,60)
\multiput(0,0)(0,0){1}{\line(1,0){80}}
\multiput(0,0)(0,0){1}{\line(0,1){60}}
\multiput(0,10)(2,0){40}{\line(0,1){.1}}
\multiput(0,20)(2,0){40}{\line(0,1){.1}}
\multiput(0,30)(2,0){40}{\line(0,1){.1}}
\multiput(0,40)(2,0){40}{\line(0,1){.1}}
\multiput(0,50)(2,0){40}{\line(0,1){.1}}
\multiput(0,60)(2,0){40}{\line(0,1){.1}}
\multiput(10,0)(0,2){30}{\line(1,0){.1}}
\multiput(20,0)(0,2){30}{\line(1,0){.1}}
\multiput(30,0)(0,2){30}{\line(1,0){.1}}
\multiput(40,0)(0,2){30}{\line(1,0){.1}}
\multiput(50,0)(0,2){30}{\line(1,0){.1}}
\multiput(60,0)(0,2){30}{\line(1,0){.1}}
\multiput(70,0)(0,2){30}{\line(1,0){.1}}
\multiput(80,0)(0,2){30}{\line(1,0){.1}}
\put(10,0){\tableau{{}}}
\put(20,30){\tableau{{}\\{}\\{}\\{}}} \put(30,10){\tableau{{}\\{}}}
\put(40,0){\tableau{{}}} \put(60,30){\tableau{{}\\{}\\{}\\{}}}
\multiput(45,5)(0,0){1}{\line(-1,1){10}}
\multiput(35,15)(0,0){1}{\line(0,-1){10}}
\multiput(35,5)(0,0){1}{\line(-1,1){10}}
\multiput(25,15)(0,0){1}{\line(0,-1){10}}
\multiput(25,5)(0,0){1}{\vector(-1,1){10}}
\end{picture}\hskip.1in
\begin{picture}(80,60)
\multiput(0,0)(0,0){1}{\line(1,0){80}}
\multiput(0,0)(0,0){1}{\line(0,1){60}}
\multiput(0,10)(2,0){40}{\line(0,1){.1}}
\multiput(0,20)(2,0){40}{\line(0,1){.1}}
\multiput(0,30)(2,0){40}{\line(0,1){.1}}
\multiput(0,40)(2,0){40}{\line(0,1){.1}}
\multiput(0,50)(2,0){40}{\line(0,1){.1}}
\multiput(0,60)(2,0){40}{\line(0,1){.1}}
\multiput(10,0)(0,2){30}{\line(1,0){.1}}
\multiput(20,0)(0,2){30}{\line(1,0){.1}}
\multiput(30,0)(0,2){30}{\line(1,0){.1}}
\multiput(40,0)(0,2){30}{\line(1,0){.1}}
\multiput(50,0)(0,2){30}{\line(1,0){.1}}
\multiput(60,0)(0,2){30}{\line(1,0){.1}}
\multiput(70,0)(0,2){30}{\line(1,0){.1}}
\multiput(80,0)(0,2){30}{\line(1,0){.1}}
 \put(10,0){\tableau{{}}}
\put(20,30){\tableau{{}\\{}\\{}\\{}}} \put(30,10){\tableau{{}\\{}}}
\put(40,0){\tableau{{}}} \put(60,30){\tableau{{}\\{}\\{}\\{}}}
\multiput(65,5)(0,0){1}{\vector(-1,1){40}}
\end{picture}
\end{center}
\end{ex}

On the other hand, the other basic algorithm, called the sliding, is a way to increase the
number of cells in a tower diagram. In this case, under certain
conditions, we can slide new cells into the diagram along reverse
diagonal lines. This operation can be thought as a sliding of
numbers into the diagram. Practically, when we slide the number $i$,
we place a new cell $c$ so that it lies on $x+y=i$ and its
east border is the interval $[i,i+1]$ on the $y$-axis.  Then we let it
slide on the line passing through its main diagonal. This can be
thought as the sliding of $i$ into the diagram. Now there are four
cases.

\begin{enumerate}
\item[(i)] If the line $x+y=i$ does not intersect any cell of the tower
diagram, we say that the cell $c$ has a slide into the tower diagram
through  $x+y=i$. The sliding stops when the cell intersects with
the $x$-axis.
\item[(ii)] Otherwise, while sliding, if the first intersection of the line
$x+y=i$ appears on the northeast corner of some cell, say $d$, then
the top of $d$ is empty and $c$ is placed on the top of $d$.
\item[(iii)] For the other case, if the first intersection of $x+y=i$ is on the north-west
corner of a cell, say $d$, and if the top of $d$ is empty, then we
say that the cell $c$ has no slide into the tower diagram, and that
the sliding algorithm terminates for the cell $c$.
\item[(iv)] In the remaining case, if the first intersection is with the
north-west corner of a cell, say $d$, and if the top of $d$ is
non-empty,  then  we move the cell $c$ one level up and let it
continue its sliding through the new diagonal line $x+y=i+1$, starting from the top of $d$,
subject to the same conditions.
\end{enumerate}

 We include the technical definition
of sliding in the appendix, where it is used to prove a technical
lemma.

\begin{ex} In the following example, we illustrate the four cases of the sliding
algorithm. Observe that the sliding algorithm terminates only on the
third case.
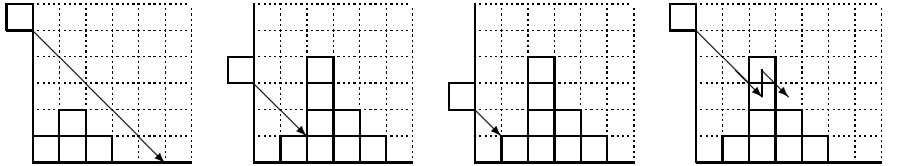
\begin{center}
\begin{picture}(80,60)
\multiput(0,0)(0,0){1}{\line(1,0){60}}
\multiput(0,0)(0,0){1}{\line(0,1){60}}
\multiput(0,10)(2,0){30}{\line(0,1){.1}}
\multiput(0,20)(2,0){30}{\line(0,1){.1}}
\multiput(0,30)(2,0){30}{\line(0,1){.1}}
\multiput(0,40)(2,0){30}{\line(0,1){.1}}
\multiput(0,50)(2,0){30}{\line(0,1){.1}}
\multiput(0,60)(2,0){30}{\line(0,1){.1}}
\multiput(10,0)(0,2){30}{\line(1,0){.1}}
\multiput(20,0)(0,2){30}{\line(1,0){.1}}
\multiput(30,0)(0,2){30}{\line(1,0){.1}}
\multiput(40,0)(0,2){30}{\line(1,0){.1}}
\multiput(50,0)(0,2){30}{\line(1,0){.1}}
\multiput(60,0)(0,2){30}{\line(1,0){.1}}
 \put(0,0){\tableau{{}}}
\put(10,10){\tableau{{}\\{}}}
\put(20,0){\tableau{{}}}
 \put(-10,50){\tableau{{}}}
\multiput(0,50)(0,0){1}{\vector(1,-1){50}}
\end{picture}\hskip 0.05in
\begin{picture}(80,60)
\multiput(0,0)(0,0){1}{\line(1,0){60}}
\multiput(0,0)(0,0){1}{\line(0,1){60}}
\multiput(0,10)(2,0){30}{\line(0,1){.1}}
\multiput(0,20)(2,0){30}{\line(0,1){.1}}
\multiput(0,30)(2,0){30}{\line(0,1){.1}}
\multiput(0,40)(2,0){30}{\line(0,1){.1}}
\multiput(0,50)(2,0){30}{\line(0,1){.1}}
\multiput(0,60)(2,0){30}{\line(0,1){.1}}
\multiput(10,0)(0,2){30}{\line(1,0){.1}}
\multiput(20,0)(0,2){30}{\line(1,0){.1}}
\multiput(30,0)(0,2){30}{\line(1,0){.1}}
\multiput(40,0)(0,2){30}{\line(1,0){.1}}
\multiput(50,0)(0,2){30}{\line(1,0){.1}}
\multiput(60,0)(0,2){30}{\line(1,0){.1}}
 \put(10,0){\tableau{{}}}
\put(20,30){\tableau{{}\\{}\\{}\\{}}}
\put(30,10){\tableau{{}\\{}}}
\put(40,0){\tableau{{}}}
\multiput(0,30)(0,0){1}{\vector(1,-1){20}}
 \put(-10,30){\tableau{{}}}
\end{picture}\hskip 0.05in
\begin{picture}(80,60)
\multiput(0,0)(0,0){1}{\line(1,0){60}}
\multiput(0,0)(0,0){1}{\line(0,1){60}}
\multiput(0,10)(2,0){30}{\line(0,1){.1}}
\multiput(0,20)(2,0){30}{\line(0,1){.1}}
\multiput(0,30)(2,0){30}{\line(0,1){.1}}
\multiput(0,40)(2,0){30}{\line(0,1){.1}}
\multiput(0,50)(2,0){30}{\line(0,1){.1}}
\multiput(0,60)(2,0){30}{\line(0,1){.1}}
\multiput(10,0)(0,2){30}{\line(1,0){.1}}
\multiput(20,0)(0,2){30}{\line(1,0){.1}}
\multiput(30,0)(0,2){30}{\line(1,0){.1}}
\multiput(40,0)(0,2){30}{\line(1,0){.1}}
\multiput(50,0)(0,2){30}{\line(1,0){.1}}
\multiput(60,0)(0,2){30}{\line(1,0){.1}}
\put(10,0){\tableau{{}}}
\put(20,30){\tableau{{}\\{}\\{}\\{}}}
\put(30,10){\tableau{{}\\{}}}
\put(40,0){\tableau{{}}}
\put(-10,20){\tableau{{}}}
\multiput(0,20)(0,0){1}{\vector(1,-1){10}}
\end{picture}\hskip 0.05in
\begin{picture}(70,60)
\multiput(0,0)(0,0){1}{\line(1,0){70}}
\multiput(0,0)(0,0){1}{\line(0,1){60}}
\multiput(0,10)(2,0){35}{\line(0,1){.1}}
\multiput(0,20)(2,0){35}{\line(0,1){.1}}
\multiput(0,30)(2,0){35}{\line(0,1){.1}}
\multiput(0,40)(2,0){35}{\line(0,1){.1}}
\multiput(0,50)(2,0){35}{\line(0,1){.1}}
\multiput(0,60)(2,0){35}{\line(0,1){.1}}
\multiput(10,0)(0,2){30}{\line(1,0){.1}}
\multiput(20,0)(0,2){30}{\line(1,0){.1}}
\multiput(30,0)(0,2){30}{\line(1,0){.1}}
\multiput(40,0)(0,2){30}{\line(1,0){.1}}
\multiput(50,0)(0,2){30}{\line(1,0){.1}}
\multiput(60,0)(0,2){30}{\line(1,0){.1}}
\multiput(70,0)(0,2){30}{\line(1,0){.1}}
 \put(10,0){\tableau{{}}}
\put(20,30){\tableau{{}\\{}\\{}\\{}}} \put(30,10){\tableau{{}\\{}}}
\put(40,0){\tableau{{}}} \multiput(0,50)(0,0){1}{\vector(1,-1){25}}
\multiput(25,35)(0,0){1}{\line(0,-1){10}}
\multiput(25,35)(0,0){1}{\vector(1,-1){10}}
\put(-10,50){\tableau{{}}}
\end{picture}
\end{center}
\end{ex}

It is easy to prove that the flight and the sliding algorithms  are
inverse to each other. We refer to \cite{CT} for detailed
definitions and examples. With the sliding algorithm, we can slide
words into the empty diagram to obtain tower diagrams, whenever the
sliding algorithm does not terminate for each letter of the word. In
this case, in order to keep track of the order of appearance of the
cells, we put numbers $1, 2,\ldots, l$ inside the cells where $l$ is
the number of cells in the tower diagram.  For example, by sliding
the word $\alpha = 54534562$ into the empty diagram, we obtain the
tower diagram $(0,1,4,2,1)$ together with the corresponding
numbering given below.

\begin{center}
\begin{picture}(100,60)
\multiput(0,0)(0,0){1}{\line(1,0){100}}
\multiput(0,0)(0,0){1}{\line(0,1){60}}
\multiput(0,10)(2,0){50}{\line(0,1){.1}}
\multiput(0,20)(2,0){50}{\line(0,1){.1}}
\multiput(0,30)(2,0){50}{\line(0,1){.1}}
\multiput(0,40)(2,0){50}{\line(0,1){.1}}
\multiput(0,50)(2,0){50}{\line(0,1){.1}}
\multiput(0,60)(2,0){50}{\line(0,1){.1}}
\multiput(10,0)(0,2){30}{\line(1,0){.1}}
\multiput(20,0)(0,2){30}{\line(1,0){.1}}
\multiput(30,0)(0,2){30}{\line(1,0){.1}}
\multiput(40,0)(0,2){30}{\line(1,0){.1}}
\multiput(50,0)(0,2){30}{\line(1,0){.1}}
\multiput(60,0)(0,2){30}{\line(1,0){.1}}
\multiput(70,0)(0,2){30}{\line(1,0){.1}}
\multiput(80,0)(0,2){30}{\line(1,0){.1}}
\multiput(90,0)(0,2){30}{\line(1,0){.1}}
\multiput(80,0)(0,2){30}{\line(1,0){.1}}
\multiput(90,0)(0,2){30}{\line(1,0){.1}} \put(10,0){\tableau{{8}}}
\put(20,30){\tableau{{7}\\{6}\\{5}\\{4}}} \put(30,10){\tableau{{3}\\{2}}}
\put(40,0){\tableau{{1}}}
\end{picture}
\end{center}

Now we call a tower tableau \textbf{\textit{standard}} if it is
obtained by the  sliding of a word. It is possible to characterize
standard tower tableaux by referring to the flight algorithm instead
of sliding. More precisely, we call a tower tableau $T$ of $n$ cells
standard, if for any $k\in \{ 0,1,\ldots, n-1 \}$, the cell numbered
with $n-k$ is a corner cell in the partial tableau $T_{\le n-k}$
obtained by forgetting all the cells with greater label.

It follows from the definition that labeling of the tower diagram $\mathcal T$ obtained by labeling the right most
bottom cell with $1$ and continuing from bottom to top and right to left is standard. This special labeling is called
the natural tower tableau of shape $\mathcal T$ and is denoted by $\mathbb T$.

In \cite[Theorem 4.3]{CT}, we show that given a word $\alpha$, the sliding algorithm produces a standard tower
tableau if and only if the word is a reduced word for some permutation. This establishes a bijective
correspondence between the set of all reduced words and the set of all standard tower tableaux. On the other
hand, by \cite[Theorem 4.4]{CT}, any permutation determines a unique shape, in other words, two different
reduced words for a given permutation determine two different labelings of the same tower diagram. As a result,
we obtain a bijective correspondence between
\begin{enumerate}
\item[(a)] the set of all finite permutations and
\item[(b)] the set of all (finite) tower diagrams.
\end{enumerate}
Combining these two bijections, for a given permutation $\omega$ with the associated tower diagram $\mathcal T
$, we obtain a bijective correspondence between
\begin{enumerate}
\item[(a)] the set of all reduced words representing $\omega$ and
\item[(b)] the set of all standard tower tableaux of shape $\mathcal T$.
\end{enumerate}
Here, given a reduced word $\alpha$, the standard tower tableau is determined by the sliding algorithm,
whereas, given a standard tower tableau, the corresponding reduced word is determined by the flight algorithm.
This goes as follows. Let $T$ be a standard tower tableau of size $l$ and for any $1\le i\le l$, let $\alpha_i$
be the flight number of the cell with label $i$ in the tableau $T_{\le i}$. Then the word $\alpha=
\alpha_1\alpha_2\ldots\alpha_l$ is reduced whose sliding gives the tableau $T$. We call $\alpha$ the
\textbf{reading word} of $T$.

\subsection{Semi-standard tower tableaux}
In order to use tower diagrams in the context of Schubert polynomials, we need to introduce semi-standard
labelings. These will also generalize the earlier definition of being standard. The definition goes as follows.

\begin{defn}
Let $\mathcal T$ be a tower diagram and let $f:\mathcal T\rightarrow \mathbb N$ be a function. Also let $n$ be the
maximum value of $f$.

\begin{enumerate}
\item The set $$T = \{ ((i,j), f(i,j))\mid (i,j)\in \mathcal T \}$$
is called a \textbf{tower tableau} of shape $\mathcal T$. In this case, we write $\mbox{\rm shape}(T)=\mathcal T$
and call $f(i,j)$ the \textbf{label} of $(i,j)$.
\item Let
$$T_{\le m} = \{ ((i,j),f(i,j))\in T\mid f(i,j) \le m\}$$
be the sub-tableau of $T$ (not necessarily a tower diagram)
consisting of the cells of $T$ with label less than or equal to $m$,
for any $m$.
\item  The tableau $T$ is called a \textbf{semi-standard tower tableau} if the following conditions are satisfied.
\begin{enumerate}
\item The set $C_{T}$ of corners with maximal label $n$ is not empty.
\item Letting $c$ be the cell in $C_{T}$ with minimal flight number, the tower tableau $c \nwarrow
T_{\le n}$ is semi-standard.
\end{enumerate}
\end{enumerate}
\end{defn}
It is easy to prove that when the function $f$ is injective, with the image $\{1,2,3,\ldots n\}$, then being semi-standard
is equivalent to being standard. Moreover, it follows from the definition that one can associate a standard tower
tableau to any given semi-standard tower tableau, in a unique way. Indeed, let $T$ be a semi-standard tower
tableau, and let
$l$ be the number of cells in $T$. Also let $\mathcal T$ be the shape of $T$. Then the standard tower tableau
$S(T)$, called the \textbf{standardization} of $T$, is defined recursively as follows.

If $l$ is equal to one, then there is a unique standard tower
tableau of this shape, and we let $S(T)$ be this unique tableau. For
$l>1$, let $c$ be the unique corner cell of $T$ with maximal label
and minimal flight number. Then we define the tableau $S(T)$ as the
standard tower tableau of shape $\mathcal T$ where the cell $c$ has
label $l$ and $$c\nwarrow S(T) =S(c\nwarrow T).$$ 
It is clear from
its construction that the standardization of a semi-standard tower
tableau is a standard tower tableau. Now we can associate two
sequences of positive integers to a semi-standard tower  tableau
$T$. One of the sequences is the reading word of the standardization
$S(T)$ of $T$, while the other one is the sequence of labels of $T$
arranged in weakly increasing order. The following example
illustrates a semi-standard tower tableau and its standardization
respectively.

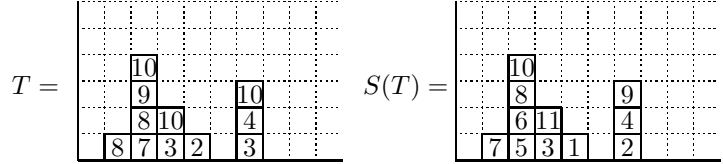
\begin{center}
\begin{picture}(100,60)
\put(-25,25){$T=$} \multiput(0,0)(0,0){1}{\line(1,0){100}}
\multiput(0,0)(0,0){1}{\line(0,1){60}}
\multiput(0,10)(2,0){50}{\line(0,1){.1}}
\multiput(0,20)(2,0){50}{\line(0,1){.1}}
\multiput(0,30)(2,0){50}{\line(0,1){.1}}
\multiput(0,40)(2,0){50}{\line(0,1){.1}}
\multiput(0,50)(2,0){50}{\line(0,1){.1}}
\multiput(0,60)(2,0){50}{\line(0,1){.1}}
\multiput(10,0)(0,2){30}{\line(1,0){.1}}
\multiput(20,0)(0,2){30}{\line(1,0){.1}}
\multiput(30,0)(0,2){30}{\line(1,0){.1}}
\multiput(40,0)(0,2){30}{\line(1,0){.1}}
\multiput(50,0)(0,2){30}{\line(1,0){.1}}
\multiput(60,0)(0,2){30}{\line(1,0){.1}}
\multiput(70,0)(0,2){30}{\line(1,0){.1}}
\multiput(80,0)(0,2){30}{\line(1,0){.1}}
\multiput(90,0)(0,2){30}{\line(1,0){.1}}
\multiput(80,0)(0,2){30}{\line(1,0){.1}}
\multiput(90,0)(0,2){30}{\line(1,0){.1}} \put(10,0){\tableau{{8}}}
\put(20,30){\tableau{{10}\\{9}\\{8}\\{7}}}
\put(30,10){\tableau{{10}\\{3}}} \put(40,0){\tableau{{2}}}
\put(60,20){\tableau{{10}\\{4}\\{3}}}
\end{picture}
\hspace{0.5in}
\begin{picture}(100,60)
\put(-35,25){$S(T)=$} \multiput(0,0)(0,0){1}{\line(1,0){100}}
\multiput(0,0)(0,0){1}{\line(0,1){60}}
\multiput(0,10)(2,0){50}{\line(0,1){.1}}
\multiput(0,20)(2,0){50}{\line(0,1){.1}}
\multiput(0,30)(2,0){50}{\line(0,1){.1}}
\multiput(0,40)(2,0){50}{\line(0,1){.1}}
\multiput(0,50)(2,0){50}{\line(0,1){.1}}
\multiput(0,60)(2,0){50}{\line(0,1){.1}}
\multiput(10,0)(0,2){30}{\line(1,0){.1}}
\multiput(20,0)(0,2){30}{\line(1,0){.1}}
\multiput(30,0)(0,2){30}{\line(1,0){.1}}
\multiput(40,0)(0,2){30}{\line(1,0){.1}}
\multiput(50,0)(0,2){30}{\line(1,0){.1}}
\multiput(60,0)(0,2){30}{\line(1,0){.1}}
\multiput(70,0)(0,2){30}{\line(1,0){.1}}
\multiput(80,0)(0,2){30}{\line(1,0){.1}}
\multiput(90,0)(0,2){30}{\line(1,0){.1}}
\multiput(80,0)(0,2){30}{\line(1,0){.1}}
\multiput(90,0)(0,2){30}{\line(1,0){.1}} \put(10,0){\tableau{{7}}}
\put(20,30){\tableau{{10}\\{8}\\{6}\\{5}}}
\put(30,10){\tableau{{11}\\{3}}} \put(40,0){\tableau{{1}}}
\put(60,20){\tableau{{9}\\{4}\\{2}}}
\end{picture}
\end{center}

Observe that $57483425964$ is the reading word of the tableau $S(T)$ whereas
$23347889(10)(10)(10)$ is the sequence of labels in $T$ ordered in
weakly increasing fashion.

\subsection{Recollections on balanced labeling}\label{Section:Balanced}
In this section, we introduce basic definitions related to balanced
labeling of Rothe diagrams. Let $\omega\in S_n$ be a permutation. The {\textbf{\it
inversion diagram}} of $\omega$  is  the diagram defined by
$$\Inv(w)=\{ (i,j) ~|~ i<j,~ \omega_i>\omega_j\} \subseteq
[n]\times [n].$$ Clearly the inversion diagram of $\omega$ encodes
the inversions of $w$ . Equivalent to $\Inv(\omega)$, we also define
the {\it Rothe diagram } $\D_\omega$ of $\omega$ as the set given by
$$
\D_\omega =\{ (i,\omega_j) ~|~ i<j,~ \omega_i>\omega_j\} \subseteq
[n]\times [n].
$$
The Rothe diagram is obtained in the following way. For any diagram $\D$, and any $(i,j) \in \D$, we define the
{\it hook $H_{(i,j)}(\D)$ with vertex} $(i,j)$ as the set
$$H_{(i,j)}(\D)= \{ (r,s) \in \D ~|~  r\geq i, s= j \}\cup \{ (r,s) \in \D ~|~  r=i, s\geq j\}.$$
Visually, we take all the cells in $\D$ which are either on the same row as $(i,j)$ and to the right of it, or in the
same column as $(i,j)$ and below it.

Now, to determine the Rothe diagram $\D_\omega$ of $\omega$, we
remove all the hooks $H_{(i,w_i)} \left( [n]\times[n]  \right)$ from
$[n]\times[n]$. The remaining cells are the cells of the Rothe
diagram and the remaining hooks  of $(i,j) \in \D_\omega$ are
denoted by  $H_{(i,j)}(\omega)$.

\begin{ex}Let $\omega = 35421$. The diagram on the left shows the $5\times 5$ array with the hooks that are to
be removed to obtain the Rothe diagram $\D_\omega$ of $\omega$ which is shown on the right.
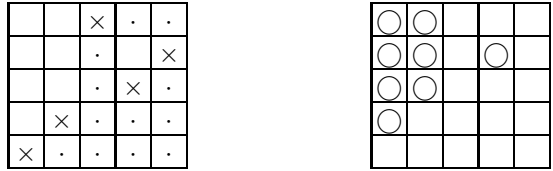
\begin{align*}
\begin{Young}
&& $\times$ &$\cdot$ &$\cdot$\cr
& &$\cdot$  & &$\times$\cr
 &&$\cdot$  &$\times$&$\cdot$\cr
 &$\times$&$\cdot$&$\cdot$ &$\cdot$\cr
$\times$&$\cdot$ &$\cdot$&$\cdot$  &$\cdot$\cr
\end{Young}&&&
\begin{Young}
$\bigcirc$&$\bigcirc$&&&\cr
$\bigcirc$&$\bigcirc$&&$\bigcirc$&\cr
$\bigcirc$&$\bigcirc$&&&\cr
$\bigcirc$&&&&\cr
&&&&\cr
\end{Young}
\end{align*}
\end{ex}

Next, recall that the right descents of the permutation $\omega$ form the set  $\Des(\omega)=\{i ~|~ \omega_i>
\omega_{i+1} \}$.
The Rothe diagram can be used to describe the set $\Des(\omega)$. More precisely, we have the following
correspondence
$$\begin{array}{ccc} \Des(\omega) & \leftrightarrow & \{(i,\omega_{i+1})\mid (i,\omega_{i+1}) \in \D_\omega\} \\
                                      i & \leftrightarrow &  (i,\omega_{i+1})\\
                                      \end{array}$$
 A cell of $\D_\omega$ that corresponds to a descent of $\omega$ under this correspondence is called a
 {\it border cell} of  $\D_\omega$. As we show in the next section, there is a close relation between border cells
 in Rothe diagrams and corner cells in tower diagrams.

Now, a labeling $T: \D_\omega \mapsto \mathbb N^+ $ is called a {\it
balanced labeling} if for every hook $H_{(i,j)}(\omega)$, the label
of the vertex cell $(i,j)$ does not change after  reordering labels
in $H_{(i,j)}(\omega)$ in a weakly decreasing way from bottom to top
and left to right. For example, the following labeling of the
diagram from the previous example is balanced.
\begin{align*}
\begin{Young}
2&1&&&\cr
4&3&&4&\cr
5&2&&&\cr
6&&&&\cr
&&&&\cr
\end{Young}
\end{align*}
Further, a balanced labeling of $\D_\omega$  is called {\it column
strict} if the numbers in each column do not repeat and is called
{\it injective} if each of the labels $\{1,2,\ldots, l\}$ appears in
$\D_\omega$  exactly once where $l$ is the size of $\D_\omega$. For
example, the above balanced labeling is column-strict but not
injective.

We denote by $\BL(\omega)$, $\IBL(\omega)$ and $\CBL(\omega)$ the set of all balanced labelings, injective
balanced labelings and column strict balanced labelings of the Rothe diagram of $\omega$, respectively.

A connection  between column strict balanced labelings  of
$\D_\omega$ and the reduced decomposition  of $\omega$ is given in
\cite{FGRS}.  Let $\omega\in S_n$ be of length $l(\omega)=k$. It is
well-known that the length $l(\omega)$ is also equal to the size of
the inversion set $\Inv(\omega)$. This implies that each
transposition corresponds to a unique inversion $(i,j)$ and hence a
cell $(i,\omega_j)$ in $\D_\omega$.

This connection is made explicit in \cite{FGRS} by associating an injective balanced labeling
\begin{equation*}\label{map-T_a}
D_\alpha:\D_\omega \mapsto [k]
\end{equation*}
of $\D_\omega$ to each reduced word $\alpha=
{\alpha_1}{\alpha_2}~\ldots~{\alpha_k}$ of $\omega$, by assigning
$r\in [k]$ to a cell $(i,j)\in\D_\omega$, if $s_{\alpha_r}$
transposes $\omega_i$ and $j$ where $j< \omega_i$. We call
$D_\alpha$ the canonical labeling of $\D_\omega$ associated to
$\alpha$. We illustrate the definition with the  diagram of the
previous example. Let $\alpha = 42341234$ be a reduced word of
$\omega$. Then the corresponding injective balanced labeling is
given as follows.
\begin{align*}
\begin{Young}
5&2&&&\cr
6&3&&1&\cr
7&4&&&\cr
8&&&&\cr
&&&&\cr
\end{Young}
\end{align*}

One of the main results in \cite{FGRS} states that the above construction induces a bijective correspondence
between Red$(\omega)$ and IBL$(\omega)$, see Theorem 2.4 in \cite{FGRS}. The inverse bijection is given by 
constructing
a reduced word from a given injective balanced labeling, in the following way. Let $D\in$IBL$(\omega)$, and 
suppose that the length of $\omega$ is $k$. Then for each
$i$ with $1\le i\le k$, we define $$\alpha_i = I(i) + R^+(i) + U^+(i)$$ where $I(i)$ denotes the row index of the
cell in $D$ with label $i$ and $R^+(i)$ (resp. $U^+(i)$) denotes the number of entries $j>i$ in the same
row as $i$ (resp. above $i$ in the same column). Then by Theorem 5.2 in \cite{FGRS}, the word $\alpha=
\alpha_1\alpha_2\ldots\alpha_k$ is a reduced word for the permutation $\omega$.

\subsection{Schubert polynomials and Stanley symmetric functions}
Finally, we recall the formula for the Schubert polynomial  $\mathfrak S_\omega$ of  $\omega \in S_n$ which is
given by Billey, Jockusch and Stanley in \cite{BJS} together with Stanley symmetric functions $F_\omega$,
 defined by Stanley in \cite{S}.

We call a pair $(a,i)$ of sequences $a=(a_1,a_2,\ldots,a_k)$ and $
i=(i_1,i_2,\ldots,i_k )$ of positive integers {\it compatible for
$\omega$} if $a_1a_2\ldots a_k$  is a reduced word for $\omega$ and
if for each $1\leq r \leq k$, we have $i_r\leq a_r$,  $i_r \leq
i_{r+1}$ and    $i_r <i_{r+1}$ whenever $a_r<a_{r+1}$. Then the
Billey-Jockusch-Stanley description of the polynomial $\mathfrak
S_\omega$ is given by
$$ \mathfrak S_{\omega}(x_1,x_2,\ldots)=\sum_{(a,i) \in \cp(w)}    x_{i_1}x_{i_2}\ldots x_{i_k}$$
where the sum is over all compatible pairs for the permutation
$\omega$, see  \cite[Theorem 1.1]{BJS}.

The result of  Fomin, Greene, Reiner and Shimozono in \cite{FGRS} shows that the
compatible pairs for $\omega$ can be determined by certain balanced
labelings of the Rothe diagram. More precisely, they show that there
is a bijective correspondence between the set of all compatible
pairs for $\omega$ and the set of all flagged column-strict balanced
labeling of $\D_\omega$. Here, flagged labeling is a labeling where
every cell in the $i$-th row is labelled by an integer less than or equal to
$i$.

In \cite{S}, Stanley defines a family of functions, the so-called Stanley symmetric functions, as the limits of
Schubert polynomials. Fomin, Greene, Reiner and Shimonozo showed in \cite[Theorem 4.3]{FGRS} that
taking the limit amounts to removing the flag conditions from the above equality.

\section{Balanced labelings via tower tableaux}\label{Section:BalancedLabel}
In this section, we establish the connection between balanced labelings and semi-standard tower tableaux and
hence the identification of being balanced via the sliding algorithm. To achieve this, we first describe how we can
remove an initial segment from a given standard tower tableaux. This will lead us to a generalization of the
Rothification algorithm described in \cite[Section 7]{CT}.

Recall that for a standard tower tableaux $T$ of size $n>0$, and for any $k$, the tower tableau $T_{\leq k}$ is
obtained by restricting its labels  to
$1,\ldots,k$, for some $1\leq k \leq n$. Moreover if $\alpha:= \alpha_1\alpha_2\ldots\alpha_n $ is the reading
word of $T$ then $\alpha_1\alpha_2\ldots\alpha_k $ is the reading word of $T_{\leq k}$ and the cell labeled by
$k$ in $T$ is a corner cell in $T_{\leq k}$. In the following  we call $T_{\leq k}$ an  \textbf{initial segment} of $T$.

It is clear that removing an initial segment $T_{\leq k}$ from a standard tower tableau $T$ reduces the diagram
$T$ to the tower diagram of the remaining reduced word $\alpha_{k+1}\ldots\alpha_n$. To be able to use this
reduction in inductive arguments, we present the following recursive algorithm.

\begin{lem}\label{Lemma:key}
Let $T=(T_1,T_2,\ldots)$ be a standard tower tableau with reading
word $\alpha = \alpha_1\alpha_2\ldots\alpha_n$ and suppose that the
cell labeled by $1$, say $c$, is contained in the tower $T_i$ for
some $ i\geq 1$. Then the standard tower tableau $S$ corresponding
to the word $\alpha^\prime = \alpha_2\ldots\alpha_n$ is obtained
from $T$ in the following steps:
\begin{enumerate}
\item First remove  $[c,1]$ from $T_i$ and push the remaining cells down in $T_i$
\item Then switch the resulting tower with  the adjacent tower $T_{i+1}$ (possibly empty)
\item Finally decrease all the labels in the resulting tableau by $1$, in order to obtained $S$.
\end{enumerate}
\end{lem}

We defer the proof to the appendix and have an example below.

\begin{ex} We illustrate the algorithm described in the lemma by the following standard tower
tableaux $T,S$ and $R$ whose  reading words are, respectively
$\tau=134534$, $\sigma=453451$ and $\rho=314354$. Observe that the
resulting tower tableaux $T',S'$ and $R'$ are in fact the recording
tableaux of $\tau'=34534$, $\sigma'=53451$ and $\rho'=14354$,
respectively, in the sliding algorithm.

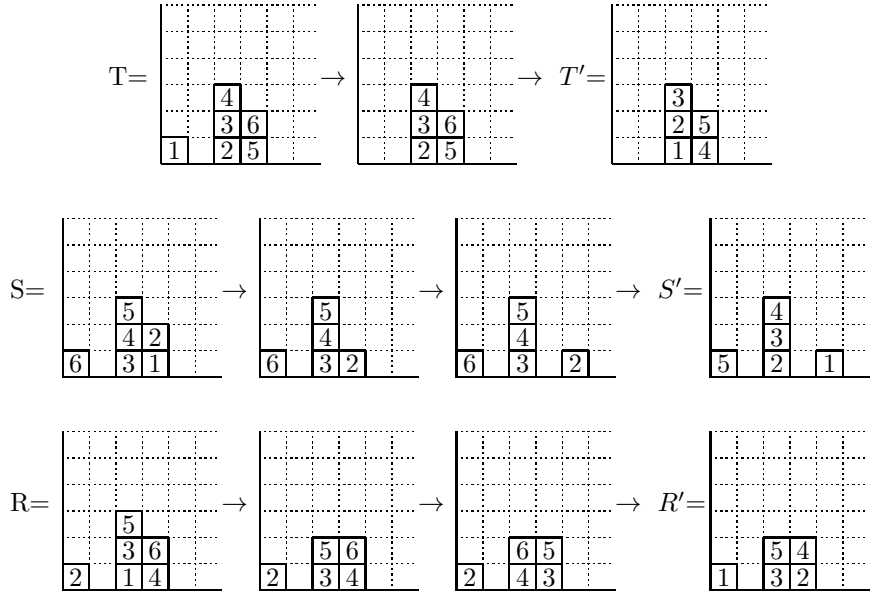
\begin{center}
\begin{picture}(60,60)
\put(-20,30){T=}
 \multiput(0,0)(0,0){1}{\line(1,0){60}}
\multiput(0,0)(0,0){1}{\line(0,1){60}}
\multiput(0,10)(2,0){30}{\line(0,1){.1}}
\multiput(0,20)(2,0){30}{\line(0,1){.1}}
\multiput(0,30)(2,0){30}{\line(0,1){.1}}
\multiput(0,40)(2,0){30}{\line(0,1){.1}}
\multiput(0,50)(2,0){30}{\line(0,1){.1}}
\multiput(0,60)(2,0){30}{\line(0,1){.1}}
\multiput(10,0)(0,2){30}{\line(1,0){.1}}
\multiput(20,0)(0,2){30}{\line(1,0){.1}}
\multiput(30,0)(0,2){30}{\line(1,0){.1}}
\multiput(40,0)(0,2){30}{\line(1,0){.1}}
\multiput(50,0)(0,2){30}{\line(1,0){.1}} \put(0,0){\tableau{{1}}}
\put(20,20){\tableau{{4}\\{3}\\{2}}} \put(30,10){\tableau{{6}\\{5}}}
\put(60,30){$\rightarrow$}
\end{picture}\hskip.2in
\begin{picture}(60,60)
\multiput(0,0)(0,0){1}{\line(1,0){60}}
\multiput(0,0)(0,0){1}{\line(0,1){60}}
\multiput(0,10)(2,0){30}{\line(0,1){.1}}
\multiput(0,20)(2,0){30}{\line(0,1){.1}}
\multiput(0,30)(2,0){30}{\line(0,1){.1}}
\multiput(0,40)(2,0){30}{\line(0,1){.1}}
\multiput(0,50)(2,0){30}{\line(0,1){.1}}
\multiput(0,60)(2,0){30}{\line(0,1){.1}}
\multiput(10,0)(0,2){30}{\line(1,0){.1}}
\multiput(20,0)(0,2){30}{\line(1,0){.1}}
\multiput(30,0)(0,2){30}{\line(1,0){.1}}
\multiput(40,0)(0,2){30}{\line(1,0){.1}}
\multiput(50,0)(0,2){30}{\line(1,0){.1}}
\put(20,20){\tableau{{4}\\{3}\\{2}}} \put(30,10){\tableau{{6}\\{5}}}
\put(60,30){$\rightarrow$}\end{picture}\hskip.5in
\begin{picture}(60,60)
\put(-20,30){$T^\prime$=}\multiput(0,0)(0,0){1}{\line(1,0){60}}
\multiput(0,0)(0,0){1}{\line(0,1){60}}
\multiput(0,10)(2,0){30}{\line(0,1){.1}}
\multiput(0,20)(2,0){30}{\line(0,1){.1}}
\multiput(0,30)(2,0){30}{\line(0,1){.1}}
\multiput(0,40)(2,0){30}{\line(0,1){.1}}
\multiput(0,50)(2,0){30}{\line(0,1){.1}}
\multiput(0,60)(2,0){30}{\line(0,1){.1}}
\multiput(10,0)(0,2){30}{\line(1,0){.1}}
\multiput(20,0)(0,2){30}{\line(1,0){.1}}
\multiput(30,0)(0,2){30}{\line(1,0){.1}}
\multiput(40,0)(0,2){30}{\line(1,0){.1}}
\multiput(50,0)(0,2){30}{\line(1,0){.1}}
\put(20,20){\tableau{{3}\\{2}\\{1}}} \put(30,10){\tableau{{5}\\{4}}}
\end{picture}
\end{center}

\vskip.1in
\begin{center}
\begin{picture}(60,60)
\put(-20,30){S=}
 \multiput(0,0)(0,0){1}{\line(1,0){60}}
\multiput(0,0)(0,0){1}{\line(0,1){60}}
\multiput(0,10)(2,0){30}{\line(0,1){.1}}
\multiput(0,20)(2,0){30}{\line(0,1){.1}}
\multiput(0,30)(2,0){30}{\line(0,1){.1}}
\multiput(0,40)(2,0){30}{\line(0,1){.1}}
\multiput(0,50)(2,0){30}{\line(0,1){.1}}
\multiput(0,60)(2,0){30}{\line(0,1){.1}}
\multiput(10,0)(0,2){30}{\line(1,0){.1}}
\multiput(20,0)(0,2){30}{\line(1,0){.1}}
\multiput(30,0)(0,2){30}{\line(1,0){.1}}
\multiput(40,0)(0,2){30}{\line(1,0){.1}}
\multiput(50,0)(0,2){30}{\line(1,0){.1}} \put(0,0){\tableau{{6}}}
\put(20,20){\tableau{{5}\\{4}\\{3}}}
\put(30,10){\tableau{{2}\\{1}}}\put(60,30){$\rightarrow$}
\end{picture}\hskip.2in
\begin{picture}(60,60)
\multiput(0,0)(0,0){1}{\line(1,0){60}}
\multiput(0,0)(0,0){1}{\line(0,1){60}}
\multiput(0,10)(2,0){30}{\line(0,1){.1}}
\multiput(0,20)(2,0){30}{\line(0,1){.1}}
\multiput(0,30)(2,0){30}{\line(0,1){.1}}
\multiput(0,40)(2,0){30}{\line(0,1){.1}}
\multiput(0,50)(2,0){30}{\line(0,1){.1}}
\multiput(0,60)(2,0){30}{\line(0,1){.1}}
\multiput(10,0)(0,2){30}{\line(1,0){.1}}
\multiput(20,0)(0,2){30}{\line(1,0){.1}}
\multiput(30,0)(0,2){30}{\line(1,0){.1}}
\multiput(40,0)(0,2){30}{\line(1,0){.1}}
\multiput(50,0)(0,2){30}{\line(1,0){.1}} \put(0,0){\tableau{{6}}}
\put(20,20){\tableau{{5}\\{4}\\{3}}} \put(30,0){\tableau{{2}}}
\put(60,30){$\rightarrow$}\end{picture}\hskip.2in
\begin{picture}(60,60)
\multiput(0,0)(0,0){1}{\line(1,0){60}}
\multiput(0,0)(0,0){1}{\line(0,1){60}}
\multiput(0,10)(2,0){30}{\line(0,1){.1}}
\multiput(0,20)(2,0){30}{\line(0,1){.1}}
\multiput(0,30)(2,0){30}{\line(0,1){.1}}
\multiput(0,40)(2,0){30}{\line(0,1){.1}}
\multiput(0,50)(2,0){30}{\line(0,1){.1}}
\multiput(0,60)(2,0){30}{\line(0,1){.1}}
\multiput(10,0)(0,2){30}{\line(1,0){.1}}
\multiput(20,0)(0,2){30}{\line(1,0){.1}}
\multiput(30,0)(0,2){30}{\line(1,0){.1}}
\multiput(40,0)(0,2){30}{\line(1,0){.1}}
\multiput(50,0)(0,2){30}{\line(1,0){.1}} \put(0,0){\tableau{{6}}}
\put(20,20){\tableau{{5}\\{4}\\{3}}}
 \put(40,0){\tableau{{2}}}
\put(60,30){$\rightarrow$}\end{picture}\hskip.5in
\begin{picture}(60,60)
\put(-20,30){$S^\prime$=}\multiput(0,0)(0,0){1}{\line(1,0){60}}
\multiput(0,0)(0,0){1}{\line(0,1){60}}
\multiput(0,10)(2,0){30}{\line(0,1){.1}}
\multiput(0,20)(2,0){30}{\line(0,1){.1}}
\multiput(0,30)(2,0){30}{\line(0,1){.1}}
\multiput(0,40)(2,0){30}{\line(0,1){.1}}
\multiput(0,50)(2,0){30}{\line(0,1){.1}}
\multiput(0,60)(2,0){30}{\line(0,1){.1}}
\multiput(10,0)(0,2){30}{\line(1,0){.1}}
\multiput(20,0)(0,2){30}{\line(1,0){.1}}
\multiput(30,0)(0,2){30}{\line(1,0){.1}}
\multiput(40,0)(0,2){30}{\line(1,0){.1}}
\multiput(50,0)(0,2){30}{\line(1,0){.1}} \put(0,0){\tableau{{5}}}
\put(20,20){\tableau{{4}\\{3}\\{2}}}
 \put(40,0){\tableau{{1}}}
\end{picture}
\end{center}

\vskip.1in

\begin{center}
\begin{picture}(60,60)
\put(-20,30){R=}
 \multiput(0,0)(0,0){1}{\line(1,0){60}}
\multiput(0,0)(0,0){1}{\line(0,1){60}}
\multiput(0,10)(2,0){30}{\line(0,1){.1}}
\multiput(0,20)(2,0){30}{\line(0,1){.1}}
\multiput(0,30)(2,0){30}{\line(0,1){.1}}
\multiput(0,40)(2,0){30}{\line(0,1){.1}}
\multiput(0,50)(2,0){30}{\line(0,1){.1}}
\multiput(0,60)(2,0){30}{\line(0,1){.1}}
\multiput(10,0)(0,2){30}{\line(1,0){.1}}
\multiput(20,0)(0,2){30}{\line(1,0){.1}}
\multiput(30,0)(0,2){30}{\line(1,0){.1}}
\multiput(40,0)(0,2){30}{\line(1,0){.1}}
\multiput(50,0)(0,2){30}{\line(1,0){.1}} \put(0,0){\tableau{{2}}}
\put(20,20){\tableau{{5}\\{3}\\{1}}}
\put(30,10){\tableau{{6}\\{4}}}\put(60,30){$\rightarrow$}
\end{picture}\hskip.2in
\begin{picture}(60,60)
\multiput(0,0)(0,0){1}{\line(1,0){60}}
\multiput(0,0)(0,0){1}{\line(0,1){60}}
\multiput(0,10)(2,0){30}{\line(0,1){.1}}
\multiput(0,20)(2,0){30}{\line(0,1){.1}}
\multiput(0,30)(2,0){30}{\line(0,1){.1}}
\multiput(0,40)(2,0){30}{\line(0,1){.1}}
\multiput(0,50)(2,0){30}{\line(0,1){.1}}
\multiput(0,60)(2,0){30}{\line(0,1){.1}}
\multiput(10,0)(0,2){30}{\line(1,0){.1}}
\multiput(20,0)(0,2){30}{\line(1,0){.1}}
\multiput(30,0)(0,2){30}{\line(1,0){.1}}
\multiput(40,0)(0,2){30}{\line(1,0){.1}}
\multiput(50,0)(0,2){30}{\line(1,0){.1}} \put(0,0){\tableau{{2}}}
\put(20,10){\tableau{{5}\\{3}}} \put(30,10){\tableau{{6}\\{4}}}
\put(60,30){$\rightarrow$}\end{picture}\hskip.2in
\begin{picture}(60,60)
\multiput(0,0)(0,0){1}{\line(1,0){60}}
\multiput(0,0)(0,0){1}{\line(0,1){60}}
\multiput(0,10)(2,0){30}{\line(0,1){.1}}
\multiput(0,20)(2,0){30}{\line(0,1){.1}}
\multiput(0,30)(2,0){30}{\line(0,1){.1}}
\multiput(0,40)(2,0){30}{\line(0,1){.1}}
\multiput(0,50)(2,0){30}{\line(0,1){.1}}
\multiput(0,60)(2,0){30}{\line(0,1){.1}}
\multiput(10,0)(0,2){30}{\line(1,0){.1}}
\multiput(20,0)(0,2){30}{\line(1,0){.1}}
\multiput(30,0)(0,2){30}{\line(1,0){.1}}
\multiput(40,0)(0,2){30}{\line(1,0){.1}}
\multiput(50,0)(0,2){30}{\line(1,0){.1}} \put(0,0){\tableau{{2}}}
\put(20,10){\tableau{{6}\\{4}}} \put(30,10){\tableau{{5}\\{3}}}
\put(60,30){$\rightarrow$}\end{picture}\hskip.5in
\begin{picture}(60,60)
\put(-20,30){$R^\prime$=}\multiput(0,0)(0,0){1}{\line(1,0){60}}
\multiput(0,0)(0,0){1}{\line(0,1){60}}
\multiput(0,10)(2,0){30}{\line(0,1){.1}}
\multiput(0,20)(2,0){30}{\line(0,1){.1}}
\multiput(0,30)(2,0){30}{\line(0,1){.1}}
\multiput(0,40)(2,0){30}{\line(0,1){.1}}
\multiput(0,50)(2,0){30}{\line(0,1){.1}}
\multiput(0,60)(2,0){30}{\line(0,1){.1}}
\multiput(10,0)(0,2){30}{\line(1,0){.1}}
\multiput(20,0)(0,2){30}{\line(1,0){.1}}
\multiput(30,0)(0,2){30}{\line(1,0){.1}}
\multiput(40,0)(0,2){30}{\line(1,0){.1}}
\multiput(50,0)(0,2){30}{\line(1,0){.1}} \put(0,0){\tableau{{1}}}
\put(20,10){\tableau{{5}\\{3}}} \put(30,10){\tableau{{4}\\{2}}}
\end{picture}
\end{center}
\end{ex}

Using this result, we can improve the Rothification algorithm which
gives the passage from the tower diagram of a permutation to its
Rothe diagram. Originally, the Rothification makes use of the natural labeling of the given
tower diagram. The next result shows that one can achieve the same
result starting with an arbitrary standard tower tableau. We shall
have an improvement by also associating a labeling to the resulting
Rothe diagram.

We first recall our notation from \cite[Section 7]{CT}. Let $T$ be a standard tower tableau of size $l$ and 
$\Upsilon= (T,T^-)$ be the
complete tower tableau corresponding to $T$. Recall that, the complete tower diagram of a permutation is
obtained by sliding a reduced word of the permutation to the first quadrant with the $x$-axis being the border and
sliding the reverse of the word to the third quadrant with the $y$-axis being the border. See the example below.

Now we construct the set
$$
I = \{ (u,v): ([u,\mathrm{label}(u)],[v,\mathrm{label}(v)])\in T\times T^-, \mathrm{label}(u)+\mathrm{label}(v) = l+1\}
$$
of all pairs of cells from the complete tower tableau $\Upsilon$
whose labels sum up to $l+1$. Then for any $(u,v)\in I$, the
vertical shadow of $u$ and the horizontal shadow of $v$ intersect at
the point $(u_1, -v_2)$ where we write $u =(u_1,u_2)$ and $v=
(v_1,v_2)$. Then the Rothification of the complete tower tableau
$\Upsilon$ associated to $T$ is the tableau
$$
R_{T} = \{ [(u_1,-v_2),\mathrm{label}(u)]: (u,v)\in I \}.
$$

\begin{ex}
To illustrate the definition, consider the following standard tower tableau.
\begin{center}
\begin{picture}(80,60)
\multiput(0,0)(0,0){1}{\line(1,0){80}}
\multiput(0,0)(0,0){1}{\line(0,1){60}}
\multiput(0,10)(2,0){40}{\line(0,1){.1}}
\multiput(0,20)(2,0){40}{\line(0,1){.1}}
\multiput(0,30)(2,0){40}{\line(0,1){.1}}
\multiput(0,40)(2,0){40}{\line(0,1){.1}}
\multiput(0,50)(2,0){40}{\line(0,1){.1}}
\multiput(0,60)(2,0){40}{\line(0,1){.1}}
\multiput(10,0)(0,2){30}{\line(1,0){.1}}
\multiput(20,0)(0,2){30}{\line(1,0){.1}}
\multiput(30,0)(0,2){30}{\line(1,0){.1}}
\multiput(40,0)(0,2){30}{\line(1,0){.1}}
\multiput(50,0)(0,2){30}{\line(1,0){.1}}
\multiput(60,0)(0,2){30}{\line(1,0){.1}}
\multiput(70,0)(0,2){30}{\line(1,0){.1}} \put(0,0){\tableau{{2}}}
\put(20,20){\tableau{{5}\\{3}\\{1}}} \put(30,10){\tableau{{6}\\{4}}}
\end{picture}
\end{center}
The reading word of the tableau is $314354$. According to the above definition, the Rothification of  the
corresponding complete tower tableau  is given as follows.
$$
\put(0,0){\line(0,-1){60}}
\put(0,0){\line(-1,0){30}}
\put(0,0){\line(0,1){30}}
\put(0,0){\line(1,0){50}}
\put(0,0){\tableau{{2}}}
\put(20,20){\tableau{{5}\\{3}\\{1}}}
\put(30,10){\tableau{{6}\\{4}}}
\put(-10,0){\tableau{\\{5}}}
\put(-20,-20){\tableau{\\{4}}} \put(-10,-20){\tableau{\\{3}}}
\put(-20,-30){\tableau{\\{2}}} \put(-10,-30){\tableau{\\{1}}}
\put(-10,-40){\tableau{\\{6}}}
\put(0,-9){$\, 2$} \put(20,-30){$\;3$} \put(30,-30){$\;4$}
\put(20,-40){$\;5$} \put(30,-40){$\;6$} \put(20,-50){$\;1$}
$$

\end{ex}

Now we have the following result, cf \cite[Theorem 7.3]{CT}.
\begin{thm}
Let $\Upsilon = (T,T^-)$ be a complete tower tableau and let
$\omega$ be the corresponding permutation. Then $\mbox{\rm
shape}(R_T) = \D_\omega$.
\end{thm}
\begin{proof}
The proof is very similar to that of Theorem 7.3 in \cite{CT}.
We include a full proof for convenience. Argue by induction on the
length $l$ of $\omega$. The case $l=1$ is trivial. Suppose $l>1$.
Let $\alpha = \alpha_1\alpha_2\ldots\alpha_l$ be the reading word of
$T$ and let $\tilde\Upsilon = (\tilde T,\tilde T^{-})$ be the complete
tower diagram of $\tilde\alpha = \alpha_1\alpha_2\ldots\alpha_{l-1}$. By the induction hypothesis, we
have $\mathrm{shape}(R_{\tilde T}) = \D_{\tilde\omega}$ where
$\tilde\omega = \omega\alpha_l$.

Here, to obtain the Rothification of the complete tower diagram of $\tilde\omega$, we use Lemma
\ref{Lemma:key}.
Assume that the cell with label $l$ in the tower diagram $T$ of $\omega$ is on the $i$-th
column and the cell with label $1$ in the virtual tower diagram $T^-$ of $\omega$ is on the row $j$.
 Then by the construction of the virtual tower diagram, we have $j= \alpha_l$.

With these notations, Lemma \ref{Lemma:key}, the Rothefication
$R_{\tilde\omega}$ of $\tilde\omega$ is obtained from $R_\omega$ by
removing the cell $(j,i)$ from $R_\omega$ and then switching the
rows $j$ and $j+1$. By the induction hypothesis, the shape of
$R_{\tilde T}$ is equal to the Rothe diagram $\D_{\tilde\omega}$ of
$\tilde\omega$. Now $\omega = \tilde\omega\alpha_l$ and $l(\omega) =
l(\tilde\omega)+1$. Thus by Proposition 6.1 in \cite{CT}, the Rothe
diagram of $\omega$ is obtained from $\D_{\tilde\omega}$ by
switching the rows $j$ and $j+1$ and adding the cell $(j,
\omega(j+1))$ to the diagram.

Finally, to show that the equality shape$(R_{T}) = \D_\omega$ holds,
we have to prove that the cell removed from the tower tableau $T$ is
the same as the one that is added at the end. In order to prove
this, it is sufficient to show that the column numbers of these
cells are the same. But by Theorem 6.2 in \cite{CT}, the equality
$\mathrm{shape}(R_\mathbb T) = \D_\omega$ holds where $\mathbb T$ is
the natural tower tableau of $\omega$. Moreover, by Theorem 4.1 in
\cite{CT}, the equality shape$(T) = $shape$ (\mathbb T)$ holds. Now
the first equality tells us that the cell added at the end is
contained in the tower diagram shape$(\mathbb T)$ whereas the second
equality tells that the tower tableau $T$ has the same shape as the
natural tower tableau $\mathbb T$. Thus, the equality
$\mathrm{shape}(R_{T}) = \D_\omega$ holds, as required.
\end{proof}
The following corollary is immediate from the above proof.
\begin{cor}\label{Corollary:corner-border}
Let $\omega$ be a permutation with its tower diagram $\mathcal T$ and its Rothe diagram
$\D$. Then there is a bijective correspondence
between the set $C_{\mathcal T}$ of corner cells of $\mathcal T$ and
the set $B_\D$ of border cells of $\D$, given by reading the column
indexes.
\end{cor}
Now we are ready to prove one of the main results of the paper. Notice that the labeling of the Rothe diagram in
the previous example is balanced. Our result shows that this is not a coincidence.
\begin{thm}\label{Thm:Rothification}
Let $\omega$ be a permutation and $\alpha=\alpha_1\alpha_2\ldots \alpha_l$ be a reduced word
representing $\omega$. Let $T_\alpha$ be the standard tower tableau
corresponding to $\alpha$ and let $D_\alpha$ be the canonical
labeling of the Rothe diagram $\D_\omega$ of $\omega$ corresponding
to $\alpha$. Then the tableau $D_\alpha$ is the Rothification of the
tower tableau $T_\alpha$, that is, $R_{T_\alpha }=D_\alpha$.
\end{thm}

\begin{proof}
We first prove that the Rothification $R_{T_\alpha}$ is balanced. We
argue by induction on the length $l$ of the permutation $\omega$. If
the length is $1$, then the claim is trivial. So assume that $l>1$
and that the claim is true for all permutations of length $l-1$.
Then by Corollary \ref{Corollary:corner-border}, the cell, say $b$,
with label $l$ in $R_{T_\alpha}$ is a border cell. Therefore, by
\cite[Theorem 4.8]{FGRS}, the tableau $R_{T_\alpha}$ is balanced if
and only if the tableau $R_{T_\alpha}\backslash b$ is balanced,
where $R_{T_\alpha}\backslash b$ denotes the deletion of the border
cell $b$ according to \cite[Lemma 4.6]{FGRS}. Moreover, in this
case, the resulting diagram is the diagram of $\omega s$ for some transposition $s$. But by the
above theorem, the resulting tableau is the Rothification of the
tower tableau obtained by removing the corresponding corner cell,
say $c$, in $T_\alpha$. Hence by the induction hypothesis, the
diagram $R_{T_\alpha} \backslash b$ is balanced, and hence
$R_{T_\alpha}$ is balanced, as required.

Now it is sufficient to show that the word given by \cite[Theorem
5.2]{FGRS} is equal to $\alpha= \alpha_1\alpha_2\ldots \alpha_l$. We
again argue by induction on the length of $\omega$ and assume the
result for all permutations of length less than $l$. Then removing
the border cell $b$ with label $l$ from $R_{T_\alpha}$, the
remaining diagram gives the word $ \alpha_1\alpha_2
\ldots\alpha_{l-1}$, by the induction hypothesis. Thus it remains to
show that the transposition $s$ corresponding to the cell $b$ is
$\alpha_l$. By \cite[Theorem 5.2]{FGRS}, $s$ is equal to the
transposition $(i,i+1)$ where $i$ is the row index of the cell $b$.
But the row index of $b$ is equal to $\alpha_l$ since it is the
first letter slid in the virtual sliding of $\alpha$.
\end{proof}
In the reverse direction, it is possible to determine the standard
tower tableau $T_\alpha$ starting with  $D_\alpha$. Indeed,  to
obtain the tableau $T_\alpha$, we push the labels of $D_\alpha$ up
to the tower diagram of $\omega$ and then rearrange the entries within each column so
that the labels are increasing on each column from bottom to top.

To prove this observation, let $b$ be the cell in $D_\alpha$ with
label $l$. Then it is a border cell by \cite{FGRS} and hence the
corresponding cell, say $c$, in  $T_\alpha$ is a corner cell.
Moreover, removing the cell $b$ from  $D_\alpha$ corresponds to the
removal of the cell $c$ from $T_\alpha$ together with the
corresponding cell from the virtual tableau $T^-_\alpha$. Thus the
result follows by induction on the length as in the previous case,
and hence we get the following theorem.
\begin{thm}\label{thm:pushUp}
Assume the notations of the previous theorem. Then the tower tableau
$T_\alpha$ can be obtained from the canonical tableau $D_\alpha$ of
the Rothe diagram by pushing all the labels  up to the corresponding
columns and rearranging them in increasing order from bottom to
top.
\end{thm}

Via this result, we have established a bijective correspondence between injective balanced labelings of a Rothe
diagram and standard tower tableaux of the corresponding shape. Indeed, by definition, any standard tower
tableau is column strict and the corresponding labeling of the Rothe diagram coincides with the given labeling of
the associated tower diagram at each column, up to a permutation of the entries of the column. This result
can further be generalized to a bijective correspondence between semi-standard tower tableaux and all
balanced labelings in the following way.

Let $T$ be a semi-standard tower tableau  and let $\alpha$ be the
reading word of the standard tableau $S(T)$. Also let $S(\Upsilon) =
(S(T), S(T)^{-})$ be the complete tower tableau of $S(T)$. Then we
define the completion of the semi-standard tableau $T$ as the double
labelled complete tower tableau $\Upsilon = (T,T^-)$  with the same
shape as $S(\Upsilon)$ where the double labeling $$f: \Upsilon
\rightarrow \mathbb N\times \mathbb N, (i,j)\mapsto
(a_{(i,j)},b_{(i,j)})$$ is given by the following rule. If $(i,j)
\in T$, then $a_{(i,j)}$ is the label of $(i,j)$ in $T$ and
$b_{(i,j)}$ is the label of it in $S(T)$. On the other hand, if
$(i,j)\in T^-$, then $a_{(i,j)}$ is the label of $(i,j)$ in $T$ and
$b_{(i,j)}$ is the label of it in $S(T)^-$.

For example, if $T$ is the semi-standard tableau

$$
\multiput(-100,0)(0,0){1}{\line(1,0){100}}
\multiput(-100,0)(0,0){1}{\line(0,1){40}}
\put(-90,0){\tableau{{\mbox{{8}}}}}
\put(-80,20){\tableau{{\mbox{{9}}}\\{\mbox{{8}}}\\{\mbox{{7}}}}}
\put(-70,10){\tableau{{\mbox{{9}}}\\{\mbox{{3}}}}}
\put(-60,0){\tableau{{\mbox{{2}}}}}
\put(-40,10){\tableau{{\mbox{{4}}}\\{\mbox{{3}}}}}
$$

then the corresponding complete tower tableau is given as follows.

$$\put(-100,0){\line(0,-1){90}} \put(-100,0){\line(-1,0){40}}
\multiput(-100,0)(0,0){1}{\line(1,0){100}}
\multiput(-100,0)(0,0){1}{\line(0,1){40}}
\put(-90,0){\tableau{{\mbox{\tiny{8,7}}}}}
\put(-80,20){\tableau{{\mbox{\tiny{9,8}}}\\{\mbox{\tiny{8,6}}}\\{\mbox{\tiny{7,5}}}}}
\put(-70,10){\tableau{{\mbox{\tiny{9,9}}}\\{\mbox{\tiny{3,3}}}}}
\put(-60,0){\tableau{{\mbox{\tiny{2,1}}}}}
\put(-40,10){\tableau{{\mbox{\tiny{4,4}}}\\{\mbox{\tiny{3,2}}}}}
\put(-110,-10){\tableau{\\{\mbox{\tiny{8,3}}}}}
\put(-120,-10){\tableau{\\{\mbox{\tiny{7,5}}}}}
\put(-130,-10){\tableau{\\{\mbox{\tiny{3,7}}}}}
\put(-140,-10){\tableau{\\{\mbox{\tiny{2,9}}}}}
\put(-110,-30){\tableau{\\{\mbox{\tiny{9,1}}}}}
\put(-120,-30){\tableau{\\{\mbox{\tiny{9,2}}}}}
 \put(-110,-40){\tableau{\\{\mbox{\tiny{8,4}}}}}
\put(-110,-60){\tableau{\\{\mbox{\tiny{3,8}}}}}
\put(-110,-70){\tableau{\\{\mbox{\tiny{4,6}}}}}
$$

The Rothification, in this setting, is again done according to the
standardization. As in the above case, the place of the cell is
determined by the Rothification applied to the standardization,
(hence by the second coordinates of the labels) and the labels are
taken from the semi-standard tableau, (hence are the first
coordinates of the labels). We illustrate the final part of the
algorithm on the same example.
$$
\put(-100,0){\line(0,-1){90}}
\put(-100,0){\line(-1,0){40}}
\multiput(-100,0)(0,0){1}{\line(1,0){100}}
\multiput(-100,0)(0,0){1}{\line(0,1){40}}
\put(-90,0){\tableau{{\mbox{\tiny{8,7}}}}}
\put(-80,20){\tableau{{\mbox{\tiny{9,8}}}\\{\mbox{\tiny{8,6}}}\\{\mbox{\tiny{7,5}}}}}
\put(-70,10){\tableau{{\mbox{\tiny{9,9}}}\\{\mbox{\tiny{3,3}}}}}
\put(-60,0){\tableau{{\mbox{\tiny{2,1}}}}}
\put(-40,10){\tableau{{\mbox{\tiny{4,4}}}\\{\mbox{\tiny{3,2}}}}}
\put(-110,-10){\tableau{\\{\mbox{\tiny{8,3}}}}}
\put(-120,-10){\tableau{\\{\mbox{\tiny{7,5}}}}}
\put(-130,-10){\tableau{\\{\mbox{\tiny{3,7}}}}}
\put(-140,-10){\tableau{\\{\mbox{\tiny{2,9}}}}}
\put(-110,-30){\tableau{\\{\mbox{\tiny{9,1}}}}}
\put(-120,-30){\tableau{\\{\mbox{\tiny{9,2}}}}}
 \put(-110,-40){\tableau{\\{\mbox{\tiny{8,4}}}}}
\put(-110,-60){\tableau{\\{\mbox{\tiny{3,8}}}}}
\put(-110,-70){\tableau{\\{\mbox{\tiny{4,6}}}}}
\put(-90,-10){\tableau{\\{8}}}
\put(-80,-10){\tableau{\\{7}}}
\put(-70,-10){\tableau{\\{3}}}
\put(-60,-10){\tableau{\\{2}}}
\put(-80,-30){\tableau{\\{9}\\{8}}}
\put(-70,-30){\tableau{\\{9}}}
\put(-40,-60){\tableau{\\{3}\\{4}}}
$$
Now we claim that the Rothification $R_T$ of a semi-standard tower
tableau $T$ is balanced. To prove this, we argue by induction on the
number of cells in $T$. Let $c$ be the corner cell in $T$ with
maximal label and minimum flight number. Then by Corollary
\ref{Corollary:corner-border}, the corresponding cell $b$ in $R_T$
is a border cell. Then by the induction hypothesis, the tableau
$R_T\backslash b$ is balanced and hence by Theorem 4.8 in
\cite{FGRS}, the tableau $R_T$ is balanced, as required. As in the
case of standard tableaux, the reverse of the above claim is also
true, that is, pushing the labels of a balanced tableau up to the
tower diagram, we will obtain a semi-standard tower tableau. The
proof of this last claim is very similar to the proof Theorem
\ref{thm:pushUp} and is left to the reader. Hence we have obtained
the following theorem.
\begin{thm}\label{thm:semistandard-balanced}
Let $\omega$ be a permutation, $\mathcal T$ be its tower diagram and
$\D$ be its Rothe diagram. Then there is a bijective correspondence
between
\begin{enumerate}
\item the set \mbox{\rm SSTT}$(\mathcal T)$ of all semi-standard tower tableaux of shape $\mathcal T$ and
\item the set \mbox{\rm BL}$(\D)$ of all balanced labelings of $\D$
\end{enumerate}
given by sending a balanced labelling to the tower tableau obtained by pushing the labels up to the diagram
$\mathcal T$ and rearranging them in non-decreasing order on columns. The inverse bijection is given by the
Rothification algorithm described above.
\end{thm}

\begin{rem}
The significance of the above result is that by passing from the Rothe diagram to the tower diagram, we replace
the condition that the labeling is balanced by the condition that the labeling is semi-standard. We find it easier to
check whether a labeling is semi-standard than to check if a labeling is balanced. The reason for this is that while
checking the condition on the tower diagram, the tower diagram is getting smaller and smaller at each step,
while this is not the case for the process on the Rothe diagram. Using this observation, we will determine the
type of labeling that corresponds to Schubert polynomials and Stanley symmetric functions, in the next
section.
\end{rem}

We finish this section with the following result on hooks on the tower diagrams. The result will give a rationale
for the above phenomenon.

\begin{defn} Let $c=(i,j)$ and $c'=(i',j')$ be two  cells in a tower diagram  $\mathcal T=(\mathcal T_1,\mathcal
T_2,\ldots)$ lying in the towers $\mathcal T_i$ and $\mathcal
T_{i'}$ respectively. Suppose that $c$ lies on the diagonal
$x+y=d$ i.e., $i+j=d$.
\begin{enumerate}
\item We say that the cell $c'$ is \textbf{adjacent to the cell $c$ from the right} if
$i'>i$, and the cell $c'$ lies on the diagonal $x+y=d+1$, and there is no
tower between $\mathcal T_i$ and $\mathcal T_{i'}$ having a
cell on the diagonal $x+y=d$. In this case we define
$$\mathrm{East}(c,\mathcal T)=\{c\}\cup \mathrm{East}(c',\mathcal
T).$$
\item We say that the cell $c'$ is \textbf{adjacent to the cell $c$ from above} if $i'=i$ and
$j'=j+1$. In this case we define
$$\mathrm{North}(c,\mathcal
T)=\{c\}\cup \mathrm{North}(c',\mathcal T).$$
\end{enumerate}
If there is no  cell in $\mathcal T$ which is adjacent to $c$ from
the right (respectively from above) then we define
$$\mathrm{East}(c,\mathcal T)=\{c\}~~ (respectively
~~\mathrm{North}(c,\mathcal T)=\{c\}).$$
\end{defn}

\begin{defn}
Let $\mathcal T$ be a tower diagram and $c=(i,j)$ be a cell in
$\mathcal T$. Then we define the \textbf{hook} $H_c$ of $\mathcal T$
with vertex $c$ as the set of all cells in $\mathcal T$ which are adjacent to $c$ from the right or above, that is,
$$H_c= \mathrm{North}(c,\mathcal T)\cup \mathrm{East}(c,\mathcal T).
$$
\end{defn}

As an illustration of a hook, we provide the following tower diagram
where  the cells other than $c$ in $H_c$ are labeled by a bullet
$\bullet$.
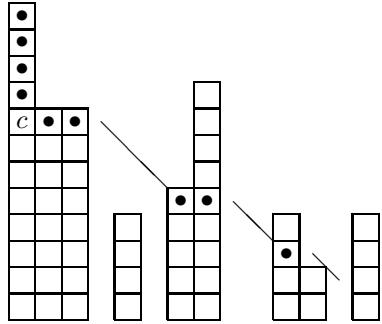
\begin{center}
\begin{picture}(100,120)
\put(0,100){\tableau{ {\bullet}\\ {\bullet}\\ {\bullet}\\ {\bullet}\\ {c}\\ {}\\ {}\\ {}\\ {}\\ {}\\ {}\\ {}}}
\put(10,60){\tableau{{\bullet}\\{}\\{}\\{}\\{}\\{}\\{}\\{}}}
\put(20,60){\tableau{{\bullet}\\{}\\{}\\{}\\{}\\{}\\{}\\{}}}
\put(40,20){\tableau{{}\\{}\\{}\\{}}}
\put(60,30){\tableau{{\bullet}\\{}\\{}\\{}\\{}}}
\put(70,70){\tableau{{}\\{}\\{}\\{}\\{\bullet}\\{}\\{}\\{}\\{}}}
\put(100,20){\tableau{{}\\{\bullet}\\{}\\ {}}}
\put(110,10){\tableau{\\{}\\ {}}}
\put(130,20){\tableau{{}\\{}\\{}\\ {}}}
\multiput(35,65)(0,0){1}{\line(1,-1){25}}
\multiput(85,35)(0,0){1}{\line(1,-1){15}}
\multiput(115,15)(0,0){1}{\line(1,-1){10}}
\end{picture}
\end{center}
\vspace{.15in}

The next result shows that our definition of a hook and the one from \cite{FGRS} are compatible. Also, as
explained above, it shows why we can replace the condition of being balanced by being semi-standard.
\begin{pro}\label{Pro:hooks}
Let $\omega$ be a permutation with tower diagram $\mathcal T$ and Rothe diagram $D$. Let
$\eta$ be its natural word with standard tower tableau $\mathbb T$ and canonical labelling $\mathbb D$ of 
$D$. Let $c$ be a cell in $\mathbb T$ and let $c''$ be the corresponding cell in $\mathbb D$.
Then, under the Rothification,
the hook $H_c$ in $\mathcal T$ coincides with the hook $H_{c''}$ in $D$.
\end{pro}

\begin{proof}
Recall that the Rothification of any two standard tower tableaux $T$ and $T'$ of shape $\mathcal T$ has the
same shape, say $\mathcal D$. Thus to prove our claim, it is sufficient to consider the natural tableau
$\mathbb{T}$ of shape $\mathcal T$. Recall also that the natural word of $\omega$ is the reading word of
$\mathbb{T}$ and it has the form
$$
\eta=\eta_n\ldots\eta_{2} \eta_1
$$
where, for each $1\leq k \leq n$, we have  $\eta_k=s_ks_{k+1}\ldots s_{k+|\mathcal T_k|-1}$ if the corresponding
tower $\mathcal T_k$ is nonempty and  otherwise $\eta_k$ is the empty word.

Let $a$ and $b$ be the cells in the hook $H_c$ of the tower diagram $\mathcal T$ which are adjacent to $c$
as illustrated in the following diagrams. Observe that, by the construction  of the hooks in a tower diagram,  it is
enough to show that under the Rothification, the corresponding cells $a''$ and $b''$ lie in the hook $H_{c''}$ of
$\mathcal D$.

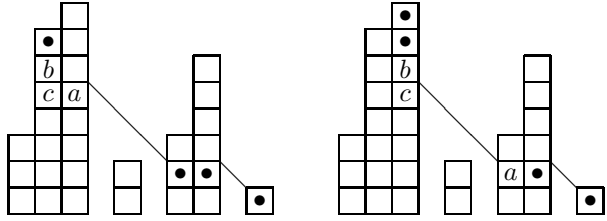
\begin{center}
\begin{picture}(100,120)
\put(0,40){\tableau{{}\\{}\\{}}}
\put(10,80){\tableau{ {\bullet}\\
{b}\\ {c}\\ {}\\ {}\\ {}\\ {}}}
\put(20,90){\tableau{{}\\{}\\{}\\{a}\\{}\\{}\\{}\\{}}}
\put(40,30){\tableau{{}\\{}}}
\put(60,40){\tableau{{}\\{\bullet}\\{}}}
\put(70,70){\tableau{{}\\{}\\{}\\{}\\{\bullet}\\{}}}
\put(90,20){\tableau{{\bullet}}}
\multiput(30,70)(0,0){1}{\line(1,-1){30}}
\multiput(80,40)(0,0){1}{\line(1,-1){10}}
\end{picture}
\hskip.3in
\begin{picture}(100,120)
\put(0,40){\tableau{{}\\{}\\{}}}
\put(10,80){\tableau{ {}\\
{}\\ {}\\ {}\\ {}\\ {}\\ {}}}
\put(20,90){\tableau{{\bullet}\\{\bullet}\\{b}\\{c}\\{}\\{}\\{}\\{}}}
\put(40,30){\tableau{{}\\{}}} \put(60,40){\tableau{{}\\{a}\\{}}}
\put(70,70){\tableau{{}\\{}\\{}\\{}\\{\bullet}\\{}}}
\put(90,20){\tableau{{\bullet}}}
\multiput(30,70)(0,0){1}{\line(1,-1){30}}
\multiput(80,40)(0,0){1}{\line(1,-1){10}}
\end{picture}
\end{center}

Now suppose that  $c$ lies on the diagonal $x+y=d$ and it  belongs to the tower $\mathcal T_i$, for some $i<n$.
Then both $a$ and $b$ lie on the diagonal $x+y=d+1$, the cell $b$ belongs to $\mathcal T_i$
and $a$ belongs to $\mathcal T_j$ for some $i<j$, subject to the condition that no tower between
$\mathcal{T}_i$ and $\mathcal{T}_j$ has a cell lying on the diagonal $x+y=d$.

Therefore in the natural word $\eta=\eta_n\ldots\eta_{j}\ldots \eta_{i}\ldots \eta_1$, we have
$$
\begin{aligned}
\eta_j&=s_j\ldots s_{d+1}\ldots s_{(j+|\mathcal T_j|-1)}\\
\eta_i&=s_i\ldots s_ds_{d+1}\ldots s_{(i+|\mathcal T_i|-1)}
\end{aligned}
$$
and, for any $i<t<j$, the maximum index that can appear in  $\eta_t$ is always less than or equal to $d-1$.
Moreover the cell $a$ is filled as a result of  the sliding of $s_{d+1}$ in $\eta_j$, whereas $c$ and $b$ are filled
as a result of the sliding $s_{d}$ and $s_{d+1}$ in $\eta_i$, respectively.

Further, let  $T'$ be the tower tableau obtained by sliding the reverse word $$\eta^r=\eta_1^r\ldots \eta_{i}^r\ldots  \eta_j^r
\ldots \eta_n^r$$ of $\eta$, where
$$
\begin{aligned}\eta_i^r&=s_{(i+|\mathcal T_i|-1)} \ldots s_{d+1}s_{d}\ldots
s_i\\
\eta_j^r&=s_{(j+|\mathcal T_j|-1)}\ldots s_{d+1}\ldots s_j.
\end{aligned}
$$

Let $b'$, $c'$ and  $a'$ be the cells in $T'$ that appear as a result of sliding, respectively, $s_{d+1}$ of $\eta_i^r
$,  $s_{d}$ of $\eta_i^r$ and  $s_{d+1}$  of $\eta_j^r$. We claim  that $b'$ is adjacent to $c'$ from the right
whereas $a'$ is  adjacent to $c'$ from above.

In order to prove the claim, first  recall  from the proof of
Lemma~7.1 in \cite{CT} that zigzag slides never occur in the sliding
of the natural word $\eta$ and its reverse $\eta^r$. In other words,
any number $t$ in these words is placed through a diagonal slide
which produces a cell lying on  the diagonal $x+y=t$.

Therefore in the sliding of $\eta^r$, the cells $b'$ and $c'$ appear as a result of sliding $s_{d+1}$ and $s_{d}$ of
$\eta_i^r$ and hence they lie  on the diagonals $x+y=d+1$ and $x+y=d$ respectively.  Since $b'$ appears first
and there are no zigzag slides,  $c'$  must be  to the left of $b'$. On the other hand, at the stage that $c'$
appears no  tower between $c'$ and $b'$ has a cell on the diagonal $x+y=d$ since otherwise $b'$ would be
placed on the top of that tower at the first place. Moreover sliding of the rest of the numbers do not create any new
cell between the towers of $c'$ and $b'$ since this necessitates zigzag slides. Therefore $b'$ is adjacent to $c'$
on the right in $T'$.

Now the cell $a'$ appears as a result of sliding of $s_{d+1}$  in $\eta_j^r$. Observe that any number between
$s_{d+1}$  of $\eta_i^r$ and $s_{d+1}$  of $\eta_j^r$ are either greater then $d+1$ or less than $d$. Since  there
are no zigzag slides, their diagonal sliding produces cells either on top or to the right of $b'$ or to the left
of $c'$. Hence   the  cell on top of $c'$, which lies on the diagonal $x+y=d+1$ remains to be  empty until the
sliding of these numbers.  Now    the sliding of  $s_{d+1}$ of $\eta_j^r$ follows the same path as the sliding of
$s_{d+1}$ of $\eta_i^r$ except that it fills the cell lying on top of $c'$, as required.

Hence we have proved that a cell $x$ in $\mathbb T$ is adjacent to $c$ from above (resp. on the right) if and only
if the corresponding cell $x'$ in $\mathbb T'$ is adjacent to the corresponding cell $c'$ on the right (resp. from
above). In other words, the hook $H_c$ of the diagram $\mathbb T$ is transferred to the hook $H_{c'}$ of the
diagram $\mathbb T'$, and vice versa.

Now, observe that the virtual tableau $\mathbb{T}^-$ of $\mathbb{T}$ is obtained by reflecting all the cells of $T'$
 along the diagonal $x+y=0$. We will denote  the resulting cells in $\mathbb{T}^-$ by the same letter as they
 appear in $T'$. By the construction of the Rothification, the cell $b''\in \mathcal D$ is obtained
 by intersecting the vertical shadow of $b \in \mathbb{T}$ and the horizontal shadow of $b'\in \mathbb{T}^-$. One
 also get $a''$ and $c''$ similarly.

Finally, we are ready to prove the proposition. First,  $b$ is adjacent to $c$ in $\mathbb T$ from above if and
only if $b'$ is adjacent to $c'$ in $\mathbb T'$ on the right if and only if $b'$ lies  in a lower row than the row of $c'$ in $\mathbb{T}^-$. But these are all equivalent to say that the cell $b''$ in $\mathbb D$ corresponding to $b$
is in the lower leg of the hook $H_{c''}$.

For the last case, $a''$ lies in the horizontal leg of the hook $H_{c''}$ of $\mathcal D$ if and only if $a'$ and $c'$
lie in the same row of  $\mathbb{T}^-$.
But this is equivalent to say that $a'$ is adjacent to $c'$ in $\mathbb T'$ on the right and hence, by the first part of
the proof, this is equivalent to say that $a$ is adjacent to $c$ on the right, as required.

\end{proof}

\section{Stanley and Schubert labellings}\label{Section:Schubert}
In this section, we characterize the type of labeling of a tower diagram that describes Schubert polynomials and
Stanley symmetric functions. We first introduce the following definition.
\begin{defn}
Let $\omega$ be a permutation and let $\mathcal T$ be a diagram associated to $\omega$.
\begin{enumerate}
\item Let $T$ be a labeling of $\mathcal T$ with labels $( 1^{a_1},  2^{a_2},\ldots, k^{a_k})$.
The \emph{reading monomial} of $T$ is the monomial
$$ x^T = x_1^{a_1}x_2^{a_2}\ldots x_k^{a_k}.$$
\item A (column strict) labeling $T$ of $\mathcal T$ is called a \emph{Stanley labeling} of $\mathcal T$ if the
reading monomial of $T$ is a monomial in the Stanley symmetric function $F_\omega$.
\item A (column strict) labeling $T$ of $\mathcal T$ is called a \emph{Schubert labeling} of $\mathcal T$ if the
reading monomial of $T$ is a monomial in the Schubert polynomial of $\omega$.
\end{enumerate}
\end{defn}
By \cite{FGRS}, any flagged balanced column strict labeling of the Rothe
diagram of $\omega$ is a Schubert labeling and by \cite{FGRS}, any balanced column strict labeling of the Rothe
diagram of $\omega$ is a Stanley labeling. Clearly, one obtains the set of Schubert labelings as a 
set of Stanley labelings by putting the flag conditions. We introduce the following definition to check flag
conditions.
\begin{defn}
Let $\omega$ be a permutation and let $\mathcal T$ be a diagram associated to $\omega$. Let $T$ and
$T^\prime$ be two labelings of $\mathcal T$. We write $T\le T^\prime$, and say that $T$ is less than or equal to
$T^\prime$, if for each cell $c\in \mathcal T$, the label of $c$ at $T$ is less than or equal to that in $T^\prime$.
\end{defn}

Next we characterize Stanley labelings of tower diagrams.
\begin{thm}
Let $\omega$ be a permutationand let $\mathcal T$ be its tower diagram. A labeling $T$ of $\mathcal T$ is
Stanley if and only if it is column strict and semi-standard.
\end{thm}
\begin{proof}
By Theorem 4.3 in \cite{FGRS}, a column strict balanced labeling of the Rothe diagram of $\omega$ is a Stanley
labeling. Now by Thereom \ref{thm:semistandard-balanced}, a tower tableau is semi-standard if and only if
its Rothification is balanced. It is also clear that a tower tableau is column-strict if and only if its Rothification is.
Therefore, the result follows.
\end{proof}

The following corollary is now immediate.
\begin{cor}
For a permutation $\omega$, we have
$$
F_\omega = \sum_T x^T
$$
where the sum is over all column strict semi-standard tower tableaux $T$ of shape $\omega$.
\end{cor}

To characterize Schubert labelings, we first introduce the flag
conditions. Let $\mathcal T$ be a tower diagram. The \textbf{flag
labeling} of $\mathcal T$ is the function $f:\mathcal T\rightarrow
\mathbb N$ such that $f(i,j) = k$ if the cell corresponding to the
cell $(i,j)$ in the Rothification is in the $k$-th row. We call the
corresponding tower tableau the \textbf{flag tableau $\mathbb F$ of
shape $\mathcal T$}. Note that one can construct the flag tableau
without referring to the Rothe diagram. The construction is as
follows.

Let $\mathcal T$ be a tower diagram. Then the flag tableau $\mathbb
F$ is obtained from $\mathcal T$ by labeling the cells in such a way
that at any step, first, among the unlabeled cells, the bottom cell
$c$ in the left most tower is labelled by its flight number. Then
any cell  in East$(c,\mathcal T)$ is labelled by the same number and
the step is over.  We give an example of this labeling which also
illustrates the steps.

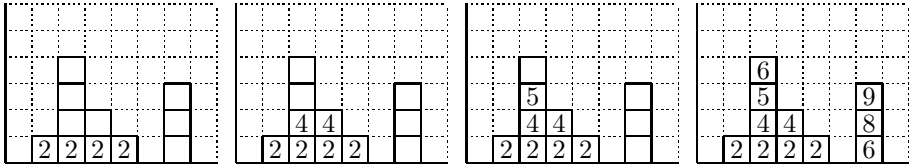
\begin{center}
\begin{picture}(80,60)
\multiput(0,0)(0,0){1}{\line(1,0){80}}
\multiput(0,0)(0,0){1}{\line(0,1){60}}
\multiput(0,10)(2,0){40}{\line(0,1){.1}}
\multiput(0,20)(2,0){40}{\line(0,1){.1}}
\multiput(0,30)(2,0){40}{\line(0,1){.1}}
\multiput(0,40)(2,0){40}{\line(0,1){.1}}
\multiput(0,50)(2,0){40}{\line(0,1){.1}}
\multiput(0,60)(2,0){40}{\line(0,1){.1}}
\multiput(10,0)(0,2){30}{\line(1,0){.1}}
\multiput(20,0)(0,2){30}{\line(1,0){.1}}
\multiput(30,0)(0,2){30}{\line(1,0){.1}}
\multiput(40,0)(0,2){30}{\line(1,0){.1}}
\multiput(50,0)(0,2){30}{\line(1,0){.1}}
\multiput(60,0)(0,2){30}{\line(1,0){.1}}
\multiput(70,0)(0,2){30}{\line(1,0){.1}}
\multiput(80,0)(0,2){30}{\line(1,0){.1}}
 \put(10,0){\tableau{{2}}}
\put(20,30){\tableau{{}\\{}\\{}\\{2}}}
\put(30,10){\tableau{{}\\{2}}}
\put(40,0){\tableau{{2}}}
\put(60,20){\tableau{{}\\{}\\{}}}
\end{picture}\hskip.1in
\begin{picture}(80,60)
\multiput(0,0)(0,0){1}{\line(1,0){80}}
\multiput(0,0)(0,0){1}{\line(0,1){60}}
\multiput(0,10)(2,0){40}{\line(0,1){.1}}
\multiput(0,20)(2,0){40}{\line(0,1){.1}}
\multiput(0,30)(2,0){40}{\line(0,1){.1}}
\multiput(0,40)(2,0){40}{\line(0,1){.1}}
\multiput(0,50)(2,0){40}{\line(0,1){.1}}
\multiput(0,60)(2,0){40}{\line(0,1){.1}}
\multiput(10,0)(0,2){30}{\line(1,0){.1}}
\multiput(20,0)(0,2){30}{\line(1,0){.1}}
\multiput(30,0)(0,2){30}{\line(1,0){.1}}
\multiput(40,0)(0,2){30}{\line(1,0){.1}}
\multiput(50,0)(0,2){30}{\line(1,0){.1}}
\multiput(60,0)(0,2){30}{\line(1,0){.1}}
\multiput(70,0)(0,2){30}{\line(1,0){.1}}
\multiput(80,0)(0,2){30}{\line(1,0){.1}}
 \put(10,0){\tableau{2{}}}
\put(20,30){\tableau{{}\\{}\\{4}\\{2}}}
\put(30,10){\tableau{{4}\\{2}}}
\put(40,0){\tableau{{2}}}
\put(60,20){\tableau{{}\\{}\\{}}}
\end{picture}\hskip.1in
\begin{picture}(80,60)
\multiput(0,0)(0,0){1}{\line(1,0){80}}
\multiput(0,0)(0,0){1}{\line(0,1){60}}
\multiput(0,10)(2,0){40}{\line(0,1){.1}}
\multiput(0,20)(2,0){40}{\line(0,1){.1}}
\multiput(0,30)(2,0){40}{\line(0,1){.1}}
\multiput(0,40)(2,0){40}{\line(0,1){.1}}
\multiput(0,50)(2,0){40}{\line(0,1){.1}}
\multiput(0,60)(2,0){40}{\line(0,1){.1}}
\multiput(10,0)(0,2){30}{\line(1,0){.1}}
\multiput(20,0)(0,2){30}{\line(1,0){.1}}
\multiput(30,0)(0,2){30}{\line(1,0){.1}}
\multiput(40,0)(0,2){30}{\line(1,0){.1}}
\multiput(50,0)(0,2){30}{\line(1,0){.1}}
\multiput(60,0)(0,2){30}{\line(1,0){.1}}
\multiput(70,0)(0,2){30}{\line(1,0){.1}}
\multiput(80,0)(0,2){30}{\line(1,0){.1}}
 \put(10,0){\tableau{2{}}}
\put(20,30){\tableau{{}\\{5}\\{4}\\{2}}}
\put(30,10){\tableau{{4}\\{2}}}
\put(40,0){\tableau{{2}}}
\put(60,20){\tableau{{}\\{}\\{}}}
\end{picture}\hskip.1in
\begin{picture}(80,60)
\multiput(0,0)(0,0){1}{\line(1,0){80}}
\multiput(0,0)(0,0){1}{\line(0,1){60}}
\multiput(0,10)(2,0){40}{\line(0,1){.1}}
\multiput(0,20)(2,0){40}{\line(0,1){.1}}
\multiput(0,30)(2,0){40}{\line(0,1){.1}}
\multiput(0,40)(2,0){40}{\line(0,1){.1}}
\multiput(0,50)(2,0){40}{\line(0,1){.1}}
\multiput(0,60)(2,0){40}{\line(0,1){.1}}
\multiput(10,0)(0,2){30}{\line(1,0){.1}}
\multiput(20,0)(0,2){30}{\line(1,0){.1}}
\multiput(30,0)(0,2){30}{\line(1,0){.1}}
\multiput(40,0)(0,2){30}{\line(1,0){.1}}
\multiput(50,0)(0,2){30}{\line(1,0){.1}}
\multiput(60,0)(0,2){30}{\line(1,0){.1}}
\multiput(70,0)(0,2){30}{\line(1,0){.1}}
\multiput(80,0)(0,2){30}{\line(1,0){.1}}
 \put(10,0){\tableau{2{}}}
\put(20,30){\tableau{{6}\\{5}\\{4}\\{2}}}
\put(30,10){\tableau{{4}\\{2}}}
\put(40,0){\tableau{{2}}}
\put(60,30){\tableau{\\{9}\\{8}\\{6}}}
\end{picture}
\end{center}
To have a comparison, we include the flagged labeled Rothe diagram
of the corresponding permutation.
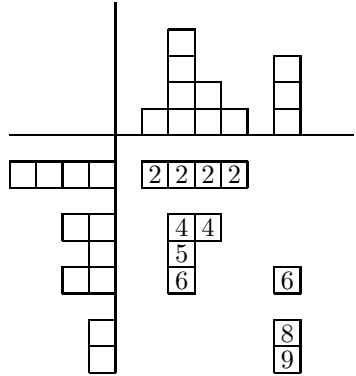
\begin{center}
\begin{picture}(120,70)
\put(0,0){\line(0,-1){90}}
\put(0,0){\line(-1,0){40}}
\put(0,0){\line(1,0){90}}
\put(0,0){\line(0,1){50}}
\put(10,0){\tableau{{}}}
\put(20,30){\tableau{{}\\{}\\{}\\{}}}
\put(30,10){\tableau{{}\\{}}}
\put(40,0){\tableau{{}}}
\put(60,20){\tableau{{}\\{}\\{}}}
\put(-10,-10){\tableau{\\{}}}
\put(-20,-10){\tableau{\\{}}}
\put(-30,-10){\tableau{\\{}}}
\put(-40,-10){\tableau{\\{}}}
\put(-10,-30){\tableau{\\{}}}
\put(-20,-30){\tableau{\\{}}}
\put(-10,-40){\tableau{\\{}}}
\put(-10,-50){\tableau{\\{}}}
\put(-20,-50){\tableau{\\{}}}
\put(-10,-70){\tableau{\\{}}}
\put(-10,-80){\tableau{\\{}}}
\put(10,-10){\tableau{\\{2}}}
\put(20,-10){\tableau{\\{2}}}
\put(30,-10){\tableau{\\{2}}}
\put(40,-10){\tableau{\\{2}}}
\put(20,-30){\tableau{\\{4}\\{5}\\{6}}}
\put(30,-30){\tableau{\\{4}}}
\put(60,-50){\tableau{\\{6}}}
\put(60,-70){\tableau{\\{8}\\{9}}}
\end{picture}
\end{center}
\vspace{1.5in}

Note that, by Proposition \ref{Pro:hooks}, the construction is equivalent to label the horizontal leg of a hook by the
row index of the corresponding hook in the Rothe diagram.  Finally, the following result is the characterization of
Schubert labeling of tower diagrams.
\begin{thm}
Let $\omega$ be a permutation and $\mathcal T$ be its tower diagram.
A labeling $T$ of $\mathcal T$ is Schubert if and only if it is
column strict, semi-standard and $T\le \mathbb F$.
\end{thm}
\begin{proof}
In this case, by Theorem 6.2 in \cite{FGRS}, a column strict flagged balanced labeling of the Rothe diagram of
$\omega$ is a Schubert labeling. In other words, a Schubert labeling is just a flagged Stanley labeling.
Now by the above observation, it is clear that a labeling of the Rothe diagram is flagged and Stanley if and only if
the corresponding tower diagram satisfies the conditions of the theorem, as required.
\end{proof}

Similar to the previous corollary, the following corollary is immediate.
\begin{cor}
For a permutation $\omega$, we have
$$
\mathfrak S_\omega = \sum_T x^T
$$
where the sum is over all column strict, semi-standard tower tableaux $T$ of shape $\omega$ with $T\le 
\mathbb F$.

\end{cor}
\appendix
\section{Proof of Lemma \ref{Lemma:key}}
\newtheorem{defnn}{Definition}
\newtheorem{lemm}{Lemma}

The proof makes use of the formal definition of the sliding algorithm which we include below.

\begin{defnn}\label{def:sliding} Let $\mathcal{T}=(\mathcal{T}_1,\mathcal{T}_2,\ldots)$
be a tower diagram and $\alpha$ be a positive integer. In the
following we denote   the \textbf{\textit{sliding of}} $\alpha$ into
$\mathcal{T}$ by
$$\alpha^{\searrow}
\mathcal{T}=\alpha^{\searrow}(\mathcal{T}_1,\mathcal{T}_2,\ldots).
$$
\begin{enumerate}
\item[\textbf{(S1)}] If  $\mathcal{T}$ has  no squares lying on the diagonal
$x+y=\alpha-1$ then we put
$$\alpha^{\searrow} \mathcal{T}:= (\mathcal{T}_1,\ldots,\mathcal{T}_{\alpha-1}) \sqcup \alpha^{\searrow}
(\mathcal{T}_{\alpha},\ldots)$$
\begin{enumerate}
\item (Direct slide) If  $\mathcal{T}$ has  no squares lying on the diagonal
$x+y=\alpha$ then necessarily  $\mathcal{T}_\alpha =\varnothing$ and
for $\mathcal{T}_\alpha'= \{ (\alpha,0)\}$
$$\alpha^{\searrow}
(\mathcal{T}_{\alpha},\ldots)=(\mathcal{T}_{\alpha}',\ldots)~~\text{and}~~
\alpha^{\searrow} \mathcal{T}:=(\mathcal{T}_1\ldots
\mathcal{T}_{\alpha-1},\mathcal{T}'_\alpha,\mathcal{T}_{\alpha+1},\ldots
).
$$
\item If $(\alpha,0) \in \mathcal{T}_\alpha$ and $(\alpha,1)\not \in
\mathcal{T}_\alpha$ then the slide  $\alpha^{\searrow} \mathcal{T}$
terminates without a result.
\item (Zigzag slide) If $(\alpha,0) \in \mathcal{T}_\alpha$ and $(\alpha,1) \in
\mathcal{T}_\alpha$  then
$$\alpha^{\searrow} \mathcal{T}:= (\mathcal{T}_1,\ldots,\mathcal{T}_\alpha) \sqcup
(\alpha+1)^{\searrow} (\mathcal{T}_{\alpha+1},\ldots).
$$
and $\alpha^{\searrow} \mathcal{T}$ terminates if and only if
$(\alpha+1)^{\searrow} (\mathcal{T}_{\alpha+1},\ldots)$ terminates.
\end{enumerate}
\item[\textbf{(S2)}]  Suppose now that  $\mathcal{T}$ has  some squares lying on the diagonal
$x+y=\alpha-1$ and  let $\mathcal{T}_i$ be the first tower  from the
left which contains such a  square, which is necessarily
$(i,\alpha-1-i)$ for some $1\leq i < \alpha$. Then we put
$$\alpha^{\searrow} \mathcal{T}:=
(\mathcal{T}_1,\ldots,\mathcal{T}_{i-1}) \sqcup \alpha^{\searrow}
(\mathcal{T}_{i},\ldots).$$
\begin{enumerate}
\item (Direct slide) If $(i,\alpha-i)\not \in \mathcal{T}_i$ then for $\mathcal{T}_i'=\mathcal{T}_i \cup \{
(i,\alpha-i)\}$,
$$\alpha^{\searrow} (\mathcal{T}_{i},\ldots):=(\mathcal{T}_i',\ldots)~~\text{and}~~
\alpha^{\searrow} \mathcal{T}:=(\mathcal{T}_1\ldots
\mathcal{T}_{i-1},\mathcal{T}'_i,\mathcal{T}_{i+1},\ldots ).
$$
\item If $(i,\alpha-i) \in \mathcal{T}_i$ and $(i,\alpha-i+1)\not \in
\mathcal{T}_i$ then the slide $\alpha^{\searrow} \mathcal{T}$
terminates without a result.
\item (Zigzag slide) If $(i,\alpha-i) \in \mathcal{T}_i$ and $(i,\alpha-i+1) \in
\mathcal{T}_i$  then
$$\alpha^{\searrow} \mathcal{T}:= (\mathcal{T}_1,\ldots,\mathcal{T}_i) \sqcup
(\alpha+1)^{\searrow} (\mathcal{T}_{i+1},\ldots)
$$
and $\alpha^{\searrow} \mathcal{T}$ terminates if and only if
$(\alpha+1)^{\searrow} (\mathcal{T}_{i+1},\ldots)$ terminates.
\end{enumerate}
\end{enumerate}

Therefore if the algorithm does not terminate then
$\alpha^{\searrow} \mathcal{T}:= \mathcal{T}\cup \{ (i,j)\}$ for
some square $(i,j)$.
\end{defnn}

The following easy result  is crucial in the proof of the lemma.

\begin{lemm} \it{Let $T=(T_1,T_2,\ldots)$ be a standard tower tableau  and
let $c$ be the cell labeled by  $1$, which is  contained in the
tower $T_i$ for some $i\geq 1$, equivalently $c=(i,0)$. Then
$|T_i|>|T_{i+1}|$.}
\end{lemm}
\begin{proof} Let $\alpha_1\alpha_2\ldots\alpha_{n}$ be the reading word of $T$.
 Observe that sliding of $\alpha_1$ produces the cell
$c=(\alpha_1,0)$ i.e., $i=\alpha_1$.
 We will prove this argument by induction on the size of the
tower tableaux. If $n=1$ then $T$ has only one cell and there is
nothing to prove. For $n=2$ we may have that  $\alpha_2<\alpha_1$,
$\alpha_2=\alpha_1+1$ or $\alpha_2>\alpha_1+1$, but  in each cases
the sliding of $\alpha_2$ after $\alpha_1$ does not produce the cell
$(i+1,0)$ therefore the statement is also true  in this case.

 Now
assume that the statement is true for all tableaux of size  $< n$.
Let $T$ be the tableau of the reading word
$\alpha_1\alpha_2\ldots\alpha_{n}$.  Recall that  $T$ is obtained by
sliding $\alpha_n$ in to the tableaux of the reading word
$\alpha_1\alpha_2\ldots\alpha_{n-1}$, namely $T_{\leq n-1}$, and
hence the size of $T$ is just one greater than that of $T_{\leq
n-1}$. On the other hand  the $i$-th tower of $T_{\leq n-1}$
contains more cells than its $(i+1)$-th tower by induction hypothesis.
Therefore   $|T_i|\geq |T_{i+1}|$ in $T$. If $|T_i|=|T_{i+1}|$ then
this shows that the cell, say $c$,  produced by the sliding of
$\alpha_n$ into $T_{\leq n-1}$ is the top cell of $T_{i+1}$, and it
is a corner cell labeled by $n$ as the following figure illustrates.
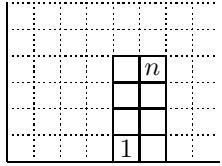
\begin{center}
\begin{picture}(80,60)
 \multiput(0,0)(0,0){1}{\line(1,0){80}}
\multiput(0,0)(0,0){1}{\line(0,1){60}}
\multiput(0,10)(2,0){40}{\line(0,1){.1}}
\multiput(0,20)(2,0){40}{\line(0,1){.1}}
\multiput(0,30)(2,0){40}{\line(0,1){.1}}
\multiput(0,40)(2,0){40}{\line(0,1){.1}}
\multiput(0,50)(2,0){40}{\line(0,1){.1}}
\multiput(0,60)(2,0){40}{\line(0,1){.1}}
\multiput(10,0)(0,2){30}{\line(1,0){.1}}
\multiput(20,0)(0,2){30}{\line(1,0){.1}}
\multiput(30,0)(0,2){30}{\line(1,0){.1}}
\multiput(40,0)(0,2){30}{\line(1,0){.1}}
\multiput(50,0)(0,2){30}{\line(1,0){.1}}
\multiput(60,0)(0,2){30}{\line(1,0){.1}}
\multiput(70,0)(0,2){30}{\line(1,0){.1}}
\put(40,30){\tableau{{}\\{}\\{}\\{1}}}
\put(50,30){\tableau{{n}\\{}\\{}\\{}}}
\end{picture}
\end{center}
As it can be easily observed from the above picture,  no matter how
the towers other than $T_i$ and $T_{i+1}$ are
 settled in $T$, the cell $c$ labeled $n$ cannot have a flight path in
 $T$, which contradicts it being a corner cell. Therefore $|T_i|> |T_{i+1}|$ in
 $T$.

\end{proof}

Now, we prove Lemma \ref{Lemma:key}.
\begin{proof} We prove this result using induction on the size $n$ of the tableaux, and hence on the
corresponding reading words.
For $n=1$, the argument is clearly true. For  $n=2$,  we have either
$|T_i|=1$ or $|T_i|=2$ but always $|T_{i+1}|=0$ and easy analysis on
the corresponding reading words gives the desired result.

Let $T$ and $T'$ be  the tableaux with reading words
$\alpha_1\alpha_2\ldots\alpha_{n}$ and $\alpha_2\ldots\alpha_{n-1}$
respectively. We have that
$$T=\alpha_n \searrow T_{\leq n-1} ~~\text{and}~~T'=\alpha_n \searrow T'_{\leq n-1}
$$
where the reading words of $T_{\leq n-1}$ and $T'_{\leq n-1}$ are,
respectively,  $\alpha_1\alpha_2\ldots\alpha_{n-1}$ and
$\alpha_2\ldots\alpha_{n-1}$. Observe that the cells labeled by $1$
in $T_{\leq n-1}$ and in $T$ are the same.  Therefore by induction
hypothesis, we may assume that $T'_{\leq n-1}$ is obtained from
$T_{\leq n-1}$ by switching its $i$-th and $(i+1)$-st towers in the
way described by the algorithm.  Observe further that the remaining
towers of $T'_{\leq n-1}$ and $T_{\leq n-1}$ are exactly the same.

Let  $c$ and $d$ be the cells produced by  $\alpha_n \searrow
T'_{\leq n-1}$ and  $\alpha_n \searrow T_{\leq n-1}$ respectively.

\noindent{\bf Case 1.} We first assume that $c$  appears before the
$i$-th tower of $T_{\leq n-1}$. Then $c=d$, since the part of
$T_{\leq n-1}$ and $T'_{\leq n-1}$ from the first tower to the
$(i-1)$-st are the same.  On the other hand the label of $c$ is $n-1$
whereas the label of $d$ is $n$. Hence $T'$ is obtained from $T$ by
the above algorithm.

\noindent{\bf Case 2.} Now we suppose that  $c$   does not appear
before the $i$-th tower of $T_{\leq n-1}$.  Recall that
$$\alpha_n \searrow T_{\leq n-1}=(T_1,\ldots,T_{i-1})\sqcup
(\alpha_n+r) \searrow (T_i,T_{i+1},\ldots)
$$
where $r\leq i-1$ represent the number of times that the sliding of
$\alpha_n$  makes a zigzag through $(T_1,\ldots,T_{i-1})$.  Let
$$a=\alpha_n+r.$$

\noindent{\bf Case 2.1.} We assume that, in order to produce $T,$ the sliding $a\searrow (T_i,T_{i+1},\ldots)$
produces a cell on the tower $T_i$ of $T_{\leq n-1}$ as illustrated below. Recall that the size of the tower $T_i$ is
strictly greater than $T_{i+1}$ in $T_{\leq n-1}$ and that  $T'_{\leq n-1}$ is obtained from $T_{\leq n-1}$ by just
interchanging its $i$-st and $(i+1)$-th towers in a specific manner, which guarantees that the size of the $i$-th tower
is less than or equal to that of the $(i+1)$st tower in  $T'_{\leq n-1}$. As the following figure illustrates, this shows
that sliding of  $a$ to $T'_{\leq n-1}$ produced a cell on top of the $i$-th tower. Therefore $T'$ is obtained from
$T$ as suggested by the Lemma.

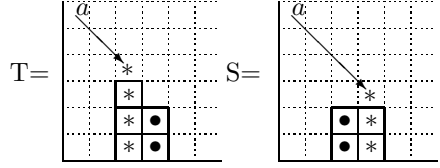
\begin{center}
\begin{picture}(60,60)
\put(-20,30){T=}
 \multiput(0,0)(0,0){1}{\line(1,0){60}}
\multiput(0,0)(0,0){1}{\line(0,1){60}}
\multiput(0,10)(2,0){30}{\line(0,1){.1}}
\multiput(0,20)(2,0){30}{\line(0,1){.1}}
\multiput(0,30)(2,0){30}{\line(0,1){.1}}
\multiput(0,40)(2,0){30}{\line(0,1){.1}}
\multiput(0,50)(2,0){30}{\line(0,1){.1}}
\multiput(0,60)(2,0){30}{\line(0,1){.1}}
\multiput(10,0)(0,2){30}{\line(1,0){.1}}
\multiput(20,0)(0,2){30}{\line(1,0){.1}}
\multiput(30,0)(0,2){30}{\line(1,0){.1}}
\multiput(40,0)(0,2){30}{\line(1,0){.1}}
\multiput(50,0)(0,2){30}{\line(1,0){.1}} \put(5,55){$a$}
\put(5,55){\vector(1,-1){18}} \put(22,32){$*$}
\put(20,20){\tableau{{*}\\{*}\\{*}}}
\put(30,10){\tableau{{\bullet}\\{\bullet}}}
\end{picture}\hskip.3in
\begin{picture}(60,60)
\put(-20,30){S=}
 \multiput(0,0)(0,0){1}{\line(1,0){60}}
\multiput(0,0)(0,0){1}{\line(0,1){60}}
\multiput(0,10)(2,0){30}{\line(0,1){.1}}
\multiput(0,20)(2,0){30}{\line(0,1){.1}}
\multiput(0,30)(2,0){30}{\line(0,1){.1}}
\multiput(0,40)(2,0){30}{\line(0,1){.1}}
\multiput(0,50)(2,0){30}{\line(0,1){.1}}
\multiput(0,60)(2,0){30}{\line(0,1){.1}}
\multiput(10,0)(0,2){30}{\line(1,0){.1}}
\multiput(20,0)(0,2){30}{\line(1,0){.1}}
\multiput(30,0)(0,2){30}{\line(1,0){.1}}
\multiput(40,0)(0,2){30}{\line(1,0){.1}}
\multiput(50,0)(0,2){30}{\line(1,0){.1}}
\put(5,55){\vector(1,-1){28}} \put(5,55){$a$} \put(32,22){$*$}
\put(30,10){\tableau{{*}\\{*}}}
\put(20,10){\tableau{{\bullet}\\{\bullet}}}
\end{picture}\hskip.2in
\end{center}

 \noindent{\bf Case 2.2.} We assume that, in order to produce $T,$ the sliding $a \searrow (T_i,T_{i+1},\ldots)$
 produces a cell on the tower $T_{i+1}$ of $T_{\leq n-1}$. Since the size of $T_i$ is greater than that of $T_{i+1}$,
 the sliding of $a$ makes a zigzag on the tower  $T_i$ as illustrated below. On the other hand it just produce a
 cell on top of the i-th tower of $T'_{\leq n-1}$. Now it is clear that $T'$ is obtain from $T$ as the lemma suggests.

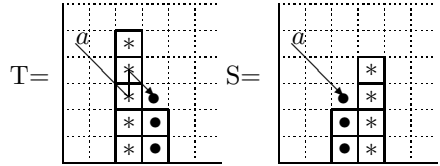
\begin{center}
\begin{picture}(60,60)
\put(-20,30){T=}
 \multiput(0,0)(0,0){1}{\line(1,0){60}}
\multiput(0,0)(0,0){1}{\line(0,1){60}}
\multiput(0,10)(2,0){30}{\line(0,1){.1}}
\multiput(0,20)(2,0){30}{\line(0,1){.1}}
\multiput(0,30)(2,0){30}{\line(0,1){.1}}
\multiput(0,40)(2,0){30}{\line(0,1){.1}}
\multiput(0,50)(2,0){30}{\line(0,1){.1}}
\multiput(0,60)(2,0){30}{\line(0,1){.1}}
\multiput(10,0)(0,2){30}{\line(1,0){.1}}
\multiput(20,0)(0,2){30}{\line(1,0){.1}}
\multiput(30,0)(0,2){30}{\line(1,0){.1}}
\multiput(40,0)(0,2){30}{\line(1,0){.1}}
\multiput(50,0)(0,2){30}{\line(1,0){.1}}
\put(20,40){\tableau{{*}\\{*}\\{*}\\{*}\\{*}}}
\put(30,10){\tableau{{\bullet}\\{\bullet}}} \put(5,45){$a$}
\put(5,45){\line(1,-1){20}} \put(25,25){\line(0,1){10}}
\put(25,35){\vector(1,-1){10}} \put(32,22){$\bullet$}
\end{picture}\hskip.3in
\begin{picture}(60,60)
\put(-20,30){S=}
 \multiput(0,0)(0,0){1}{\line(1,0){60}}
\multiput(0,0)(0,0){1}{\line(0,1){60}}
\multiput(0,10)(2,0){30}{\line(0,1){.1}}
\multiput(0,20)(2,0){30}{\line(0,1){.1}}
\multiput(0,30)(2,0){30}{\line(0,1){.1}}
\multiput(0,40)(2,0){30}{\line(0,1){.1}}
\multiput(0,50)(2,0){30}{\line(0,1){.1}}
\multiput(0,60)(2,0){30}{\line(0,1){.1}}
\multiput(10,0)(0,2){30}{\line(1,0){.1}}
\multiput(20,0)(0,2){30}{\line(1,0){.1}}
\multiput(30,0)(0,2){30}{\line(1,0){.1}}
\multiput(40,0)(0,2){30}{\line(1,0){.1}}
\multiput(50,0)(0,2){30}{\line(1,0){.1}}
\put(30,30){\tableau{{*}\\{*}\\{*}\\{*}}}
\put(20,10){\tableau{{\bullet}\\{\bullet}}}
\put(5,45){\vector(1,-1){20}} \put(5,45){$a$} \put(22,22){$\bullet$}
\end{picture}\hskip.2in
\end{center}

\noindent{\bf Case 2.3.} We assume that  $a \searrow
(T_i,T_{i+1},\ldots)$ produces a cell on a tower which is on the
right of $T_{i+1}$.

\noindent{\bf Case 2.3.1.} First assume that the sliding  $a \searrow (T_i,T_{i+1},\ldots)$ makes no zigzag on
$T_{i}$. Let $(i,j)$ be the top cell of $T_i$ of $T_{\leq n-1}$. Then $a\geq (i+j)+2$. Now in $T'_{\leq n-1}$ the
$(i+1)$-st tower is longer than $i$-th tower and since the top cell of this longer tower  is $(i+1,j-1)$,  we  see
that the sliding of $a$ to $T'_{\leq n-1}$ does not  go through the $i$-th and $(i+1)$-st towers. Since the rest of the
towers of both $T_{\leq n-1}$ and  $T'_{\leq n-1}$ are the same, the sliding of $a$ ends up in the
same way for both tableaux. Therefore $T'$ is obtained from  $T$ in the way that the lemma suggests.

 
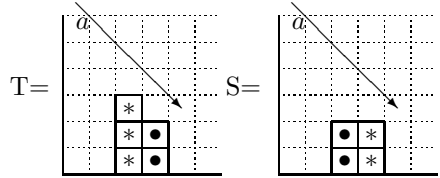
\begin{center}
\begin{picture}(60,60)
\put(-20,30){T=}
 \multiput(0,0)(0,0){1}{\line(1,0){60}}
\multiput(0,0)(0,0){1}{\line(0,1){60}}
\multiput(0,10)(2,0){30}{\line(0,1){.1}}
\multiput(0,20)(2,0){30}{\line(0,1){.1}}
\multiput(0,30)(2,0){30}{\line(0,1){.1}}
\multiput(0,40)(2,0){30}{\line(0,1){.1}}
\multiput(0,50)(2,0){30}{\line(0,1){.1}}
\multiput(0,60)(2,0){30}{\line(0,1){.1}}
\multiput(10,0)(0,2){30}{\line(1,0){.1}}
\multiput(20,0)(0,2){30}{\line(1,0){.1}}
\multiput(30,0)(0,2){30}{\line(1,0){.1}}
\multiput(40,0)(0,2){30}{\line(1,0){.1}}
\multiput(50,0)(0,2){30}{\line(1,0){.1}}
\put(20,20){\tableau{{*}\\{*}\\{*}}}
\put(30,10){\tableau{{\bullet}\\{\bullet}}} \put(5,55){$a$}
\put(5,65){\vector(1,-1){40}}
\end{picture}\hskip.3in
\begin{picture}(60,60)
\put(-20,30){S=}
 \multiput(0,0)(0,0){1}{\line(1,0){60}}
\multiput(0,0)(0,0){1}{\line(0,1){60}}
\multiput(0,10)(2,0){30}{\line(0,1){.1}}
\multiput(0,20)(2,0){30}{\line(0,1){.1}}
\multiput(0,30)(2,0){30}{\line(0,1){.1}}
\multiput(0,40)(2,0){30}{\line(0,1){.1}}
\multiput(0,50)(2,0){30}{\line(0,1){.1}}
\multiput(0,60)(2,0){30}{\line(0,1){.1}}
\multiput(10,0)(0,2){30}{\line(1,0){.1}}
\multiput(20,0)(0,2){30}{\line(1,0){.1}}
\multiput(30,0)(0,2){30}{\line(1,0){.1}}
\multiput(40,0)(0,2){30}{\line(1,0){.1}}
\multiput(50,0)(0,2){30}{\line(1,0){.1}}
\put(30,10){\tableau{{*}\\{*}}}
\put(20,10){\tableau{{\bullet}\\{\bullet}}} \put(5,55){$a$}
\put(5,65){\vector(1,-1){40}}
\end{picture}\hskip.2in
\end{center}

 \noindent{\bf Case
2.3.2.} We assume that the sliding  $a \searrow (T_i,T_{i+1},\ldots)$ makes a zigzag on $T_{i}$.  Then the sliding
of $a$ makes a zigzag at the tower $T_{i+1}$ or not as the following figures illustrates.
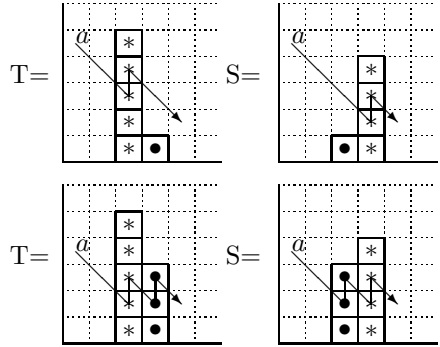
\begin{center}
\begin{picture}(60,60)
\put(-20,30){T=}
 \multiput(0,0)(0,0){1}{\line(1,0){60}}
\multiput(0,0)(0,0){1}{\line(0,1){60}}
\multiput(0,10)(2,0){30}{\line(0,1){.1}}
\multiput(0,20)(2,0){30}{\line(0,1){.1}}
\multiput(0,30)(2,0){30}{\line(0,1){.1}}
\multiput(0,40)(2,0){30}{\line(0,1){.1}}
\multiput(0,50)(2,0){30}{\line(0,1){.1}}
\multiput(0,60)(2,0){30}{\line(0,1){.1}}
\multiput(10,0)(0,2){30}{\line(1,0){.1}}
\multiput(20,0)(0,2){30}{\line(1,0){.1}}
\multiput(30,0)(0,2){30}{\line(1,0){.1}}
\multiput(40,0)(0,2){30}{\line(1,0){.1}}
\multiput(50,0)(0,2){30}{\line(1,0){.1}}
\put(20,40){\tableau{{*}\\{*}\\{*}\\{*}\\{*}}}
\put(30,0){\tableau{{\bullet}}} \put(5,45){$a$}
\put(5,45){\line(1,-1){20}} \put(25,25){\line(0,1){10}}
\put(25,35){\vector(1,-1){20}}
\end{picture}\hskip.3in
\begin{picture}(60,60)
\put(-20,30){S=}
 \multiput(0,0)(0,0){1}{\line(1,0){60}}
\multiput(0,0)(0,0){1}{\line(0,1){60}}
\multiput(0,10)(2,0){30}{\line(0,1){.1}}
\multiput(0,20)(2,0){30}{\line(0,1){.1}}
\multiput(0,30)(2,0){30}{\line(0,1){.1}}
\multiput(0,40)(2,0){30}{\line(0,1){.1}}
\multiput(0,50)(2,0){30}{\line(0,1){.1}}
\multiput(0,60)(2,0){30}{\line(0,1){.1}}
\multiput(10,0)(0,2){30}{\line(1,0){.1}}
\multiput(20,0)(0,2){30}{\line(1,0){.1}}
\multiput(30,0)(0,2){30}{\line(1,0){.1}}
\multiput(40,0)(0,2){30}{\line(1,0){.1}}
\multiput(50,0)(0,2){30}{\line(1,0){.1}}
\put(30,30){\tableau{{*}\\{*}\\{*}\\{*}}}
\put(20,0){\tableau{{\bullet}}} \put(5,45){$a$}
\put(5,45){\line(1,-1){30}} \put(35,15){\line(0,1){10}}
\put(35,25){\vector(1,-1){10}}
\end{picture}\hskip.2in
\end{center}

\begin{center}
\begin{picture}(60,60)
\put(-20,30){T=}
 \multiput(0,0)(0,0){1}{\line(1,0){60}}
\multiput(0,0)(0,0){1}{\line(0,1){60}}
\multiput(0,10)(2,0){30}{\line(0,1){.1}}
\multiput(0,20)(2,0){30}{\line(0,1){.1}}
\multiput(0,30)(2,0){30}{\line(0,1){.1}}
\multiput(0,40)(2,0){30}{\line(0,1){.1}}
\multiput(0,50)(2,0){30}{\line(0,1){.1}}
\multiput(0,60)(2,0){30}{\line(0,1){.1}}
\multiput(10,0)(0,2){30}{\line(1,0){.1}}
\multiput(20,0)(0,2){30}{\line(1,0){.1}}
\multiput(30,0)(0,2){30}{\line(1,0){.1}}
\multiput(40,0)(0,2){30}{\line(1,0){.1}}
\multiput(50,0)(0,2){30}{\line(1,0){.1}}
\put(20,40){\tableau{{*}\\{*}\\{*}\\{*}\\{*}}}
\put(30,20){\tableau{{\bullet}\\{\bullet}\\{\bullet}}}
\put(5,35){$a$} \put(5,35){\line(1,-1){20}}
\put(25,15){\line(0,1){10}} \put(25,25){\line(1,-1){10}}
\put(35,15){\line(0,1){10}} \put(35,25){\vector(1,-1){10}}
\end{picture}\hskip.3in
\begin{picture}(60,60)
\put(-20,30){S=}
 \multiput(0,0)(0,0){1}{\line(1,0){60}}
\multiput(0,0)(0,0){1}{\line(0,1){60}}
\multiput(0,10)(2,0){30}{\line(0,1){.1}}
\multiput(0,20)(2,0){30}{\line(0,1){.1}}
\multiput(0,30)(2,0){30}{\line(0,1){.1}}
\multiput(0,40)(2,0){30}{\line(0,1){.1}}
\multiput(0,50)(2,0){30}{\line(0,1){.1}}
\multiput(0,60)(2,0){30}{\line(0,1){.1}}
\multiput(10,0)(0,2){30}{\line(1,0){.1}}
\multiput(20,0)(0,2){30}{\line(1,0){.1}}
\multiput(30,0)(0,2){30}{\line(1,0){.1}}
\multiput(40,0)(0,2){30}{\line(1,0){.1}}
\multiput(50,0)(0,2){30}{\line(1,0){.1}}
\put(30,30){\tableau{{*}\\{*}\\{*}\\{*}}}
\put(20,20){\tableau{{\bullet}\\{\bullet}\\{\bullet}}}
\put(5,35){$a$} \put(5,35){\line(1,-1){20}}
\put(25,15){\line(0,1){10}} \put(25,25){\line(1,-1){10}}
\put(35,15){\line(0,1){10}} \put(35,25){\vector(1,-1){10}}
\end{picture}\hskip.2in
\end{center}

It is easy to observe  that the sliding of $a$ into both tableaux
ends at the same cell since the lengths of the towers in the rest of the tableaux are the same in
both tableaux. Therefore, the result is proved.
\end{proof}

\end{document}